\documentclass[12pt]{amsart}

\usepackage{amscd, amsfonts, amsmath, amssymb, amsthm, mathrsfs, appendix, enumerate, graphicx, flexisym, epsfig, float, graphicx, hyperref, pdfpages, tikz, mathtools, comment, tikz-cd,pinlabel, multicol}
\usepackage[all,cmtip]{xy}
\usepackage{pinlabel}
\usepackage{graphicx}

\usepackage{ulem}
\makeatletter
\@namedef{subjclassname@2020}{Mathematics Subject Classification \textup{2020}}

\newsavebox{\@brx}
\newcommand{\llangle}[1][]{\savebox{\@brx}{\(\m@th{#1\langle}\)}%
  \mathopen{\copy\@brx\kern-0.5\wd\@brx\usebox{\@brx}}}
\newcommand{\rrangle}[1][]{\savebox{\@brx}{\(\m@th{#1\rangle}\)}%
  \mathclose{\copy\@brx\kern-0.5\wd\@brx\usebox{\@brx}}}

\newtheorem{theorem}{Theorem}[section]

\newtheorem{lemma}[theorem]{Lemma}
\newtheorem{proposition}[theorem]{Proposition}
\theoremstyle{definition}

\newtheorem{definition}[theorem]{Definition}
\newtheorem{example}[theorem]{Example}
\newtheorem{remark}[theorem]{Remark}

\numberwithin{equation}{subsection}
\newtheorem*{ack}{Acknowledgement}
\usepackage[all,cmtip]{xy}

\newcommand{\interior}{\operatorname{int}}

\newcommand{\T}{\operatorname{T}}

\newcommand{\Homeo}{\operatorname{Homeo}}

\newcommand{\Mod}{\mathrm{MCG}}

\newcommand{\id}{\mathrm{id}}

\begin{document}
\title{Orderability of big mapping class groups}

\author{Pravin Kumar}
\email{pravin444enaj@gmail.com}
\address{Department of Mathematical Sciences, Indian Institute of Science Education and Research (IISER) Mohali, Sector 81,  S. A. S. Nagar, P. O. Manauli, Punjab 140306, India.}

\author{Apeksha Sanghi}
\email{apekshasanghi93@gmail.com}
\address{Department of Mathematical Sciences, Indian Institute of Science Education and Research (IISER) Mohali, Sector 81,  S. A. S. Nagar, P. O. Manauli, Punjab 140306, India.}

\author{Mahender Singh}
\email{mahender@iisermohali.ac.in}
\address{Department of Mathematical Sciences, Indian Institute of Science Education and Research (IISER) Mohali, Sector 81,  S. A. S. Nagar, P. O. Manauli, Punjab 140306, India.}



\begin{abstract}
We give an alternate proof of the left-orderability of the mapping class group of a connected oriented infinite-type surface with a non-empty boundary. Our main strategy involves the inductive construction of a countable stable Alexander system for the surface using a carefully chosen exhaustion by finite-type subsurfaces.  In fact, we prove that a generalised ideal arc system for the surface also induces a left-ordering on the big mapping class group. We then prove that two generalised ideal arc systems  determine the same left-ordering if and only if they are loosely isotopic. Finally, we prove that the  topology on the big mapping class group  is the same as the order topology induced by a left-ordering corresponding to an inductively constructed ideal arc system.
\end{abstract}
\maketitle

\section{Introduction}
The existence of a left, right or both-sided strict total ordering on a group has profound implications on its structure. For example, a left-orderable group cannot have torsion, and a bi-orderable group cannot have even generalised torsion (where a product of conjugates of a non-trivial element is trivial). From applications point of view, it is known that integral group rings of left-orderable groups are free of zero-divisors. Many groups arising in topology are left-orderable; for instance, the fundamental group of any connected surface, except for the projective plane or the Klein bottle, is bi-orderable  \cite{MR2141698}. Braid groups are notable examples of left-orderable groups that are not bi-orderable \cite{MR1214782}, whereas pure braid groups are bi-orderable  \cite{MR0975081}. In \cite{MR1756636}, Rourke and Wiest extended this result by showing that the mapping class group of a compact surface with non-empty boundary is left-orderable,  though it is generally not bi-orderable.
\par
Extensive research has been conducted on the orderability of 3-manifold groups, where left-orderability is quite common. Specifically, the fundamental groups of the complements of links in $\mathbb{S}^3$ are known to be left-orderable \cite{MR2141698}, but not all are bi-orderable \cite{MR1990838}. For example, the knot group of the figure-eight knot is bi-orderable, whereas the group of a non-trivial cable of any knot is not. Generally, a fibered knot has a bi-orderable knot group if all the roots of its Alexander polynomial are real and positive \cite{MR1990838}. There are infinitely many such fibered knots. For further reading, the recent monograph \cite{MR3560661} by Clay and Rolfsen explores the orderability of groups motivated by topology, such as fundamental groups of surfaces or 3-manifolds, braid and mapping class groups, and groups of homeomorphisms. Another monograph \cite{MR2463428} on the orderability of braid groups is also highly recommended.
\par

Dehornoy's ordering of the braid group was reinterpreted in \cite{MR1725462} in more geometrical terms. This construction was then generalised in \cite{MR1756636} to prove that the mapping class group of a compact surface with non-empty boundary is left-orderable. In \cite{MR2172491}, using hyperbolic geometry, Calegari proved the left-orderability of the mapping class group of the disk with any compact, totally disconnected subset removed. This result has been generalised by Feller, Hubbard and Turner \cite{bigsmall} by establishing left-orderability of the mapping class group of any infinite-type surface $S$ with non-empty boundary. They consider the set (which is uncountable) of homotopy classes of all essential arcs on $S$ starting from a fixed point on $\partial S$ and lift them to the universal cover of $S$. Their key idea is to define a strict total ordering on the set of these lifted arcs, which induces a mapping class group invariant strict total ordering on the original set of arcs on $S$. This strict total ordering is then used to define a left-ordering on the mapping class group of $S$. In this paper, 
we provide an alternative proof of this result using an inductively constructed countable stable Alexander system for the surface $S$ composed of non-isotopic disjoint ideal arcs, which is achieved through a carefully chosen exhaustion of $S$ by finite-type subsurfaces. Additionally, we compare the left-orderings induced by two (generalised) ideal arc systems, and use our construction to prove that the quotient topology on the big mapping class group coincides with the order topology induced by one of these orderings.
\par

The paper is organised as follows. In Section \ref{section preliminaries}, we recall some basic terminology and results that we need in latter sections. In Section \ref{Orderability of big MCG}, we prove that if $S$ is a connected oriented infinite-type surface with non-empty boundary, then its mapping class group $\Mod(S)$ is left-orderable (Theorem \ref{main theorem}).  This is achieved using an inductively constructed countable stable Alexander system for $S$. In Section \ref{section comparison of orderings}, we introduce a generalised ideal arc system for $S$ and prove that it also induces a left-ordering on $\Mod(S)$ (Proposition \ref{gen ideal arc implies left-orderability}). We then examine conditions under which two such ideal arc systems induce the same left-ordering. In fact, in Theorem \ref{thm:loose2} we prove that two generalised ideal arc systems for $S$ determine the same left-ordering on $\Mod(S)$ if and only if they are loosely isotopic.  As a consequence, we deduce that if $S$ is a connected oriented infinite genus surface with non-empty boundary, the the space of conjugacy classes of left-orderings on $\Mod(S)$ is infinite (Proposition \ref{space of orderings}). Finally, in Section \ref{section comparison of topologies}, we compare the quotient topology on the big mapping class group with the order topology induced by a left-ordering. We prove that the quotient topology on $\Mod(S)$ is the same as the order topology induced by an ordering $<_\Gamma$ corresponding to an ideal arc system $\Gamma$ that we construct in Section \ref{Orderability of big MCG} (Theorem \ref{quotient topology is order induced}).
\medskip

\section{Preliminaries}\label{section preliminaries}
We begin with some basic definitions.
 \begin{definition}
A group $G$ is called {\it left-orderable} if its elements can be given a strict total ordering $<$ which is left invariant, that is, $g<h$ implies that $f g<f h$ for all $f, g, h \in G$.  Similarly, $G$ is called {\it right-orderable} if its elements can be given a strict total ordering $<$ which is right invariant, that is, $g<h$ implies that $ gf<h f$ for all $f, g, h \in G$. Further, $G$ is called {\it bi-orderable} if there is a strict total ordering on $G$ that is simultaneously left as well as right invariant.
\end{definition}

It is easy to see that a left-orderable group can be turned into a right-orderable group with respect to a different ordering and vice-versa. Many interesting groups that are central to topology are left-orderable. For instance, free groups, braid groups \cite{MR1214782}, mapping class groups of punctures surfaces with non-empty boundary, and fundamental groups of some 3-manifolds including knot groups are left-orderable \cite{MR2141698}.
\par

 The following observation is well-known.

  \begin{proposition} 
 A group $G$ is left-orderable if and only if there exists a subset $P$ of $G$ with the following properties: 
\begin{enumerate}
     \item $P P\subset P$. 
     \item For every $g\in G$, exactly one of $g=1, g\in P \text{ or } g^{-1}\in P$ holds.
\end{enumerate} 
    \end{proposition}

In fact, given a left-ordering $<$ on a group $G$, we can take $P=\{g \in G \mid g > 1\}$, called the {\it positive cone} of the ordering.
\medskip

A surface is said to be of {\it finite-type} if its fundamental group is finitely generated; otherwise it is said  to be of {\it infinite-type}. Throughout the paper, our primary surface $S$ under consideration will be connected oriented infinite-type with non-empty boundary, unless stated otherwise. At some occasions, we shall also need finite-type surfaces, and we write $S_{\mathtt{g}, n}^b$ to denote a connected oriented finite-type surface of genus $\mathtt{g}$ with $n$ punctures and $b$ boundary components.
\par
 The {\it mapping class group}  $\Mod(S)$ of a surface $S$ (of finite or infinite-type) is the group of isotopy classes of orientation-preserving self-homeomorphisms of $S$, which preserve the boundary of $S$ point-wise. The mapping class group of an infinite-type surface is also referred as the {\it big mapping class group}. Let $\operatorname{Homeo}^{+}(S, \partial S )$ be the group of orientation-preserving self-homeomorphisms of $S$  that fix the boundary $\partial S$ point-wise, equipped with the compact-open topology. The mapping class group $\Mod(S)$ can then be equipped with the {\it quotient topology} inherited from $\operatorname{Homeo}^{+}(S,  \partial S)$, which turns it into a topological group. It is not difficult to see that $\Mod(S)$ has the discrete topology if and only if $S$ is of finite-type. Following \cite{MR2850125}, for a simple closed curve $c$ on $S$, let $T_c$ denote the left-handed Dehn twist along $c$. Further, we shall use the same notation for an orientation-preserving self-homeomorphism $f \in  \operatorname{Homeo}^{+}(S, \partial S )$, and its mapping class in $\Mod(S)$. We refer the reader to \cite{MR4264585} for a survey on both topological and algebraic aspects of big mapping class groups, and refer to \cite{MR2850125} for the general theory of mapping class groups.
\par

Let $S$ be an infinite-type surface with $b \geq 1$ boundary components and $\mathcal{E}(S)$ be its space of ends. Then $\mathcal{E}(S)= \mathcal{E}_p(S) \sqcup \mathcal{E}_{np}(S)$, where  $\mathcal{E}_p(S)$  and $\mathcal{E}_{np}(S)$ are the spaces of planar and non-planar ends, respectively. We view the set of all isolated points of $\mathcal{E}_p(S)$ as marked points on $S$ and denote this set by $P$.

\begin{definition}
Let $S$ be a surface with non-empty boundary.
\begin{enumerate}
\item An {\it ideal arc} on $S$ is the image of a continuous map 
$$
h:(I, \, \partial(I), \, I^\circ) \rightarrow \big(S, \, \partial S \cup P, \, S \setminus (\partial S \cup P) \big),
$$
which is injective on the interior $I^\circ$ of $I=[0,1]$.
\item Two ideal arcs $\gamma, \delta$ on $S$ are {\it isotopic} if there exists an isotopy of $S$ fixing $\partial S \cup P$ set-wise that deforms $\gamma$ onto $\delta$.
\end{enumerate}
\end{definition}

Throughout, we assume that an ideal arc has the canonical orientation given by its parametrisation as a map from $I$ to $S$. Also, by abuse of notation, we shall denote an arc and its image by the same notation depending on the context. 

\begin{definition}
A set $\Gamma=\{\gamma_i\}_{i \in I}$ of essential simple closed curves and arcs on a surface $S$ is said to be an {\it Alexander system} if it satisfies the following conditions:
\begin{enumerate}
\item The elements of $\Gamma$ are in pairwise minimal positions, that is, they attain the geometric intersection number of their corresponding isotopy classes.
\item If $\gamma_i, \gamma_j \in \Gamma$ for $i \neq j$, then $\gamma_i$ is not isotopic to $\gamma_j$.
\item For distinct $i, j, k \in I$, at least one of $\gamma_i \cap \gamma_j$, $\gamma_j \cap \gamma_k$ or $\gamma_k \cap \gamma_i$ is empty. 
\end{enumerate}
\end{definition}

Note that any subset of an Alexander system is again an Alexander system. The following result is proved in \cite[Proposition 2.8]{MR2850125} for finite-type surfaces and in \cite[Lemma 3.2]{MR4029627} for infinite-type surfaces. 

\begin{lemma}\label{lemma:finalex}
Let $S$ be a connected oriented surface (of finite or infinite-type) with possibly non-empty boundary, and $\Gamma$ be a finite Alexander system for $S$. Let $f \in \operatorname{Homeo}^{+}(S, \partial S)$ be such that $f(\gamma)$ is isotopic to $\gamma$ for all $\gamma \in \Gamma$. Then there exists $h \in \operatorname{Homeo}^{+}(S, \partial S)$ such that $h$ is isotopic to the identity on $S$ relative to $\partial S $ and $h|_{\gamma}=f|_{\gamma}$ for all $\gamma \in \Gamma$.
\end{lemma}

Next, we define a stable Alexander system.

\begin{definition}
Let $S$ be a connected oriented surface with possibly non-empty boundary. A set  $\Gamma$ of essential simple closed curves and arcs on $S$ is called a 
{\it stable Alexander system} for $S$ if the following conditions hold:
\begin{enumerate}
\item $\Gamma$ is an Alexander system for $S$.
\item If $f \in \Homeo^{+}(S,\partial S)$ preserves the isotopy classes of elements of $\Gamma$, then $f$ is isotopic to the identity map, relative to $\partial S$. 
\end{enumerate}
\end{definition}

\begin{definition}
Let $S$ be a surface with non-empty boundary.
\begin{enumerate}
\item An {\it ideal arc system} for $S$ is a  set $\Gamma$ of non-isotopic disjoint ideal arcs which form a stable Alexander system for $S$. 
\item Two ideal arc systems $\Gamma$ and $\Delta$ for $S$ are said to be {\it equivalent} with respect to a subset $\Sigma$ of $S$ if there is an isotopy of $S$ fixing $\partial S \cup P$ point-wise which leaves $\Sigma$ invariant and carries $\Gamma$ onto $\Delta$.
\item Two ideal arc systems $\Gamma$ and $\Delta$ for $S$ are called {\it transverse} if every arc of $\Gamma$ either coincides with some arc of $\Delta$, or it intersects the arcs of $\Delta$ transversely.
\item A {\it $D$-disk} between transverse ideal arc systems $\Gamma$ and $\Delta$ is a subset of $S$ which is homeomorphic to a closed disk without punctures in its interior, and which is bounded by a segment of an ideal arc of $\Gamma$ and a segment of an ideal arc of  $\Delta$.
\item Two ideal arc systems $\Gamma$ and $\Delta$ for $S$ are said to be {\it reduced} with respect to each other if the following conditions hold:
\begin{itemize}
\item If $\gamma \in \Gamma$ and $\delta \in \Delta$ are such that $\gamma$ and $\delta$ are isotopic, then $\gamma=\delta$.
\item  There is no $D$-disk between $\Gamma$ and $\Delta$.
\end{itemize}
\item Two sets $\mathcal{C}_1$ and $\mathcal{C}_2$ of ideal arcs for $S$ are called {\it totally disjoint} if the {\it geometric intersection number} $i(\gamma ,\delta) =0$ for each $\gamma \in \mathcal{C}_1$  and $\delta \in \mathcal{C}_2$.  
\end{enumerate}
\end{definition}

The following results are proved in \cite[Proposition 1.1 and Proposition 1.2]{MR1756636} for finite-type surfaces. Since the properties under consideration are local,  the results hold for infinite-type surfaces as well.

\begin{proposition} \label{prop:tight}
Any two ideal arc systems for a surface can be reduced with respect to each other upto isotopy. Moreover, if two ideal arc systems $\Gamma$ and $\Delta$ are both transverse to another ideal arc system $\Sigma$ and also reduced with respect to $\Sigma$, then they are equivalent with respect to $\Sigma$. 
\end{proposition}

\begin{proposition}\label{prop:reduction}
Suppose that $\Gamma, \Delta$ and $\Sigma$ are three ideal arc systems for $S$ such that $\Gamma$ and $\Delta$ are both reduced with respect to $\Sigma$. Then there exists an ideal arc system $\Gamma'$ which is isotopic to $\Gamma$ with respect to $\Sigma$, such that the three ideal arc systems $\Gamma', \Delta$ and $\Sigma$ are pairwise reduced.
\end{proposition}

\begin{definition}
Given a set  $\Gamma$ of simple arcs on a surface $S$, the surface obtained by cutting $S$ along $\cup_{\gamma\in \Gamma} \gamma$ is a surface, denoted by $S \setminus \Gamma$, satisfying the following conditions:
\begin{enumerate}
\item For each  $\gamma \in \Gamma$, there exist simple arcs $\alpha_\gamma$ and $\beta_\gamma$ on $\partial(S \setminus \Gamma)$ together with a homeomorphism $h_\gamma$  from $\alpha_\gamma$ to $\beta_\gamma$.
\item The quotient space $(S \setminus \Gamma)/\sim$ is homeomorphic to $S$, where $x \sim h_\gamma(x)$ for each $\gamma \in \Gamma$ and each $x \in \alpha_\gamma$.
\item The image of the arc $\alpha_\gamma$ (equivalently that of $\beta_\gamma$) under the quotient map is the arc $\gamma$ on $S$.
\end{enumerate} 
Similarly, we can define the surface obtained by cutting $S$ along a set consisting of simple closed curves and arcs.
\end{definition}
\bigskip

\section{Orderability of big mapping class groups}\label{Orderability of big MCG}
In this section, we prove the left-orderability of $\Mod(S)$. This is achieved through a stable Alexander system for $S$, which we construct first. For simplicity of notation, if $S'$ is a subsurface of a surface $S$, then we denote $\overline{S \setminus  S'}$ by $S\setminus S'$.

\begin{proposition} \label{prop:cmpexh} (Exhaustion)
Let $S$ be a connected oriented infinite-type surface with $b \ge 0$ boundary components. Then there exists a sequence $\{S_i \}_{i \ge 1}$ of finite-type connected subsurfaces of $S$ satisfying the following conditions:
\begin{enumerate}
\item $S_i \subset \interior(S_j)$ whenever $i < j$.
\item $S = \cup_{i=1}^{\infty} S_i$.
\item Each boundary component of $S$ is a boundary component of $S_i$ for each $i$, that is, $\partial S \subset \partial S_i$ for each $i$.
\item Each component of $\partial S_i \setminus \partial S$ is a separating simple closed curve in $S$ such that each connected component of $S \setminus S_i$ is an infinite-type subsurface.
\item For each $i \ge 1$, no component of $S_{i+1} \setminus S_i$ is an annulus.
\end{enumerate}
\end{proposition}

\begin{proof}
It is a well-known result of Rad\'{o}  \cite{Rado} that every surface is triangulable. As a topological consequence of this result (see  \cite{arXiv:2103.16702} or \cite{MR4029627}), it follows that any infinite-type surface $S$ admits a sequence $\{S_i' \}_{i \ge 1}$ of finite-type connected subsurfaces of $S$ satisfying the following conditions:
\begin{enumerate}
\item $S_i' \subset \interior(S_j')$ whenever $i < j$.
\item $S= \cup_{i=1}^{\infty} S_i'$.
\item Each boundary component of $S$ is a boundary component of $S_1'$, and hence of $S_i'$ for each $i$.
\end{enumerate}
It remains to show that the subsurfaces $\{S_i'\}$ can be modified such that the desired conditions (4) and (5) also hold. For each $i$, let $B_i = \{ b_{i,j} \mid 1 \leq j \leq b_i\}$ be the set of simple closed curves on $S_i'$ which are the components of $\partial S_i' \setminus \partial S$. Since $S_i'$ is a finite-type subsurface of $S$, the union $\cup_{j=1}^{b_i} b_{i,j}$ is separating. In general,  $S \setminus S_i'$ may have a connected component which is of finite-type. In that case, we simply remove suitable curves from $B_i$ (and denote the resulting collection also by $B_i$) such that  $S_i'$ is of finite-type and each connected component of  $S \setminus S_i'$ is of infinite-type. Suppose that $S \setminus S_i'$ has $r_i$ connected components. For each $1 \leq \ell \leq r_i$, let $B_{i,\ell}$ be the subset of $B_i$ consisting of curves which are the boundary components of the $\ell$-th connected component of  $S \setminus S_i'$.  Note that each $B_i= \cup_{\ell=1}^ {r_i } B_{i, \ell}$. We now construct our subsurfaces $\{S_i \}_{i \ge 1}$ as follows.  For each $1 \leq \ell \leq r_i$, we take a new separating curve $b_{i,\ell}'$ on  $S_i^'$ such that one of the resulting subsurface of $S_i'$ obtained by cutting it along the curve $b_{i,\ell}'$ is a 2-sphere with $|B_{i,\ell}| +1$ boundary components. Choose $S_i$ to be the subsurface of $S_i'$ such that the set of components of $\partial S_i \setminus \partial S$ is $\{b_{i,\ell}' \mid 1 \leq \ell \leq r_i\}$. See Figure \ref{fig:3} for an illustration. Thus, the sequence of subsurfaces $\{S_i \}_{i \ge 1}$ satisfies the  properties (1)-(4). Let $i$ be the smallest index such that $S_{i+1} \setminus S_i$ has a component that is an annulus. By condition (4), we can choose the smallest index $j>i+1$ such that $S_j \setminus S_i$ has no component that is an annulus. We then replace our sequence with the new sequence (after renumbering) in which the subsurfaces indexed $i+1, \ldots, j-1$ have been removed. Iterating the procedure for this new sequence, if necessary, we arrive at the desired sequence of subsurfaces.
\begin{figure}[H]
\labellist
	\tiny
	\pinlabel $b_1$ at 275, 160
	\pinlabel $b_2$ at 276, 119
	\pinlabel $b_3$ at 255, 55
	\pinlabel $b_{1,1}'$ at 200, 140
	\pinlabel $b_{1,2}'$ at 200, 50
	\endlabellist
	\includegraphics[width=2in]{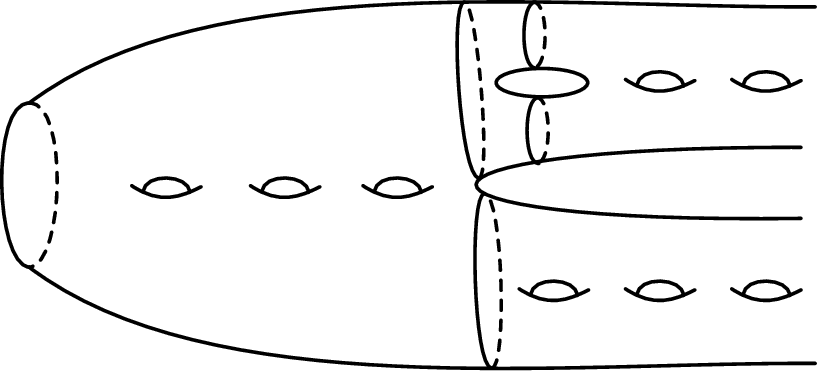}
	\caption{The set $B_1 = \{ \{b_1,b_2\}, \{b_3\} \}$.}
	\label{fig:3}
\end{figure}
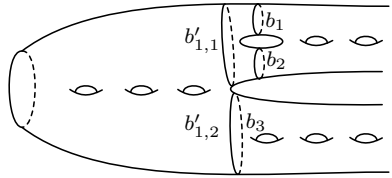
\end{proof}

For the rest of this section, we assume that $S$ is a connected oriented infinite-type surface with $b \geq 1$ boundary components.  We fix the sequence of subsurfaces $\{S_i\}_{i \ge 1}$ of $S$ as in Proposition \ref{prop:cmpexh}. Let $B_i$ denote the set of boundary components of $S_i$ that are not the boundary components of $S$, and let $b_i= |B_i|$. For each $i$, since each curve in $B_i$ is separating, it follows that the sequence $\{b_i \}_{i \ge 1}$ is non-decreasing. For each $i$, let $\widehat{S_i}$ be the surface obtained from $S_i$ by capping the boundary components from $B_i$ with disks. Then, each $\widehat{S_i}$ is a finite-type surface with exactly $b$ boundary components and is homeomorphic to $S_{g_i, p_i}^b$ for some $g_i, p_i\ge 0$. Note that, $\{g_i \}_{i \ge 1}$ and $\{p_i \}_{i \ge 1}$ are also non-decreasing sequences. 

\begin{definition}
A {\it curve diagram} for a finite-type surface $S_{g, p}^b$ is an ideal arc system $\Gamma$ for $S_{g, p}^b$ such that $S_{g, p}^b  \setminus \Gamma$ is precisely one disk without any punctures.
\end{definition}

\begin{figure}[H]
	\includegraphics[width=2.5in]{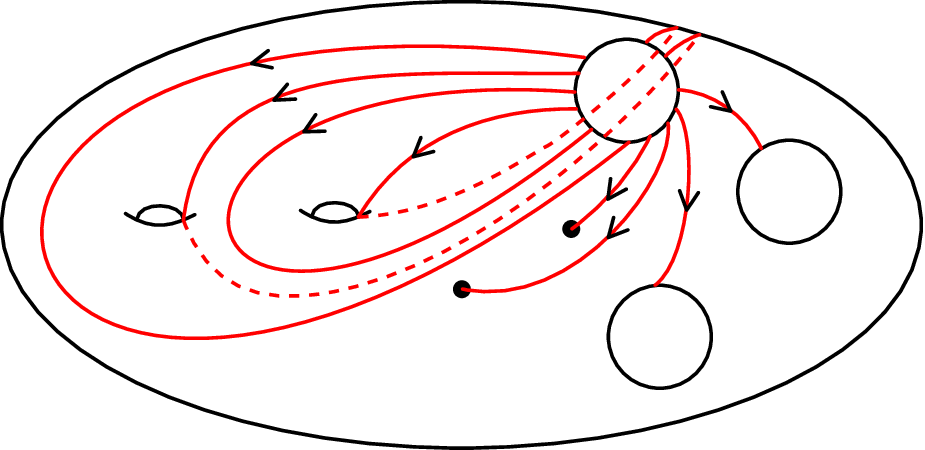}
	\caption{Example of a curve diagram for the surface $S_{2,2}^3$.}
	\label{fig:2}
\end{figure} 

Next, we shall be choosing a curve diagram for each $\widehat{S_i}$ with some desired properties. We shall then define an ideal arc system for $S$ by using these curve diagrams for $\widehat{S_i}$'s.

\begin{itemize}
	\item Note that, the subsurface $S_1$ has $b+b_1$ boundary components.
	\begin{enumerate}
\item  Let $\Gamma_1^{(1)}$ be a curve diagram for $\widehat{S_1}$ disjoint from curves in $B_1$. For example, see Figure  \ref{fig:2} for $\widehat{S_1}=S_{2,2}^3$.

\item We now define the set $\Gamma_1^{(2)}$ of $b_1-1$ separating ideal arcs on $S_1$ such that
\begin{itemize}
\item  $\Gamma_1^{(2)}$ is totally disjoint from $\Gamma_1^{(1)}$,
\item Each connected component of $S \setminus \Gamma_1^{(2)}$ contains precisely one connected component of $S \setminus S_1$. 
\end{itemize}
We construct such a $\Gamma_1^{(2)}$ as follows. Note that $S_1\setminus \Gamma_1^{(1)}$ is the 2-sphere with $1+b_1$ boundary components. If $b_1=1$, then we define $\Gamma_1^{(2)}= \emptyset$. For $b_1 \ge 2$, we proceed as follows. Let $q_1: S_1\setminus \Gamma_1^{(1)} \to S_1$ be the quotient map. Choose a set $\mathcal{C}_1$ of $b_1-1$ ideal arcs on $S_1\setminus \Gamma_1^{(1)}$ such that their end points lie on $q_1^{-1}(\partial(S) \setminus \Gamma_1^{(1)})$ (see Figure \ref{fig4} for $b_1 = 5$). Note that cutting $S_1\setminus \Gamma_1^{(1)}$ along arcs of $\mathcal{C}_1$ gives a disjoint union of $b_1$ cylinders. Let us now define $\Gamma_1^{(2)}= \{q_1(\gamma) \mid \gamma \in \mathcal{C}_1 \}$ and $\Gamma_1=\Gamma_1^{(1)} \cup \Gamma_1^{(2)}$.
		\begin{figure}[H]
			\labellist
			\tiny
			\pinlabel $b_1=5$ at 380, 180
			\endlabellist
		\includegraphics[width=2in]{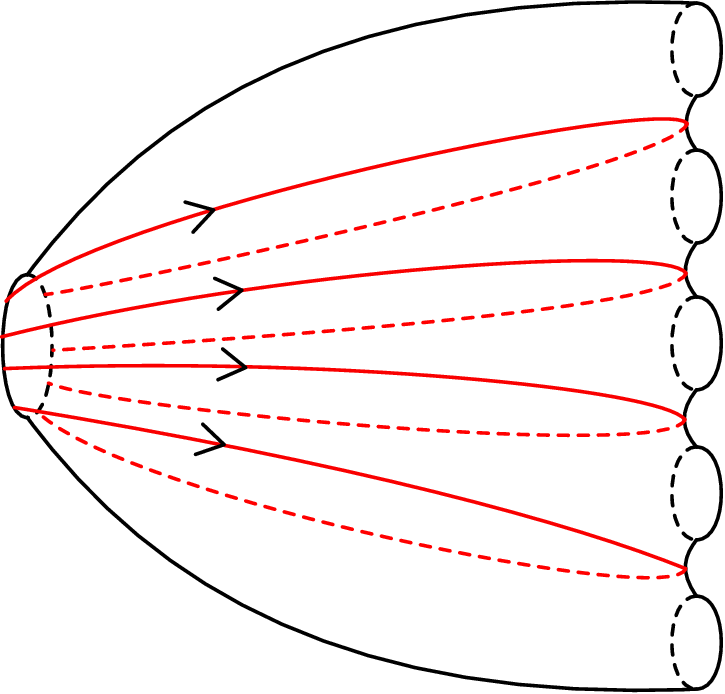}
		\caption{The set $\mathcal{C}_1$ which has 4 ideal arcs on $S_0^6$.}		
		\label{fig4}
		\end{figure}	
	\end{enumerate}		
	\item The subsurface $S_2$ has $b+b_2$ boundary components.
	\begin{enumerate}
	\item Since $S_1 \subset S_2$, we choose a curve diagram $\Gamma_2^{(1)}$ for $\widehat{S}_2$ disjoint from curves in $B_2$ such that $\Gamma_1^{(1)} \subset \Gamma_2^{(1)}$, and $\Gamma_2^{(1)}$ and $\Gamma_1^{(2)}$ are totally disjoint. See Figure \ref{fig:arc} for an illustration.
	\begin{figure}[H]
\includegraphics[width=2.5in]{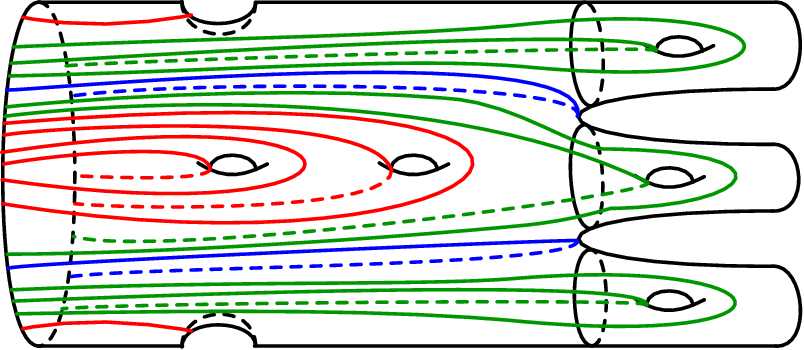}	\caption{The arcs of $\Gamma_1^{(1)}$ are in red color, the arcs of $\Gamma_1^{(2)}$ are in blue color, and the arcs of $\Gamma_2^{(1)}$ are in green color.}
\label{fig:arc}
	\end{figure}
	\item We now define the set $\Gamma_2^{(2)}$ of $b_2-1$ separating ideal arcs on $S_2$ such that
\begin{itemize}
\item $\Gamma_2^{(2)}$  is totally disjoint from $\Gamma_2^{(1)}$,
\item $\Gamma_1^{(2)} \subset \Gamma_2^{(2)}$,
\item Each connected component of $S \setminus \Gamma_2^{(2)}$ contains precisely one connected component of $S \setminus S_2$. 
\end{itemize}
We construct such a $\Gamma_2^{(2)}$ as follows.  Note that, $S_2\setminus \Gamma_2^{(1)}$ is a 2-sphere with $1+b_2$ boundary components. Let $q_2: S_2\setminus \Gamma_2^{(1)} \to S_2$ be the quotient map. Observe that $S_2\setminus (\Gamma_2^{(1)}\cup \Gamma_1^{(2)})$ is a disjoint union of $b_1$ many 2-spheres with non-empty boundary. Repeating the procedure of step (2) for the case of $S_1$ for each such 2-sphere with non-empty boundary, we obtain a set $\mathcal{C}_2$ of ideal arcs  on $S_2\setminus (\Gamma_2^{(1)}\cup \Gamma_1^{(2)})$ such that their end points lie on $q_2^{-1}(\partial(S) \setminus (\Gamma_2^{(1)} \cup \Gamma_1))$ (see Figure~\ref{fig:C2}). We define  $\Gamma_2^{(2)}= q_2(\mathcal{C}_2) \cup \Gamma_1^{(2)}$ and $\Gamma_2= \Gamma_2^{(1)} \cup \Gamma_2^{(2)}$.
\begin{figure}[H]
\includegraphics[width=2in]{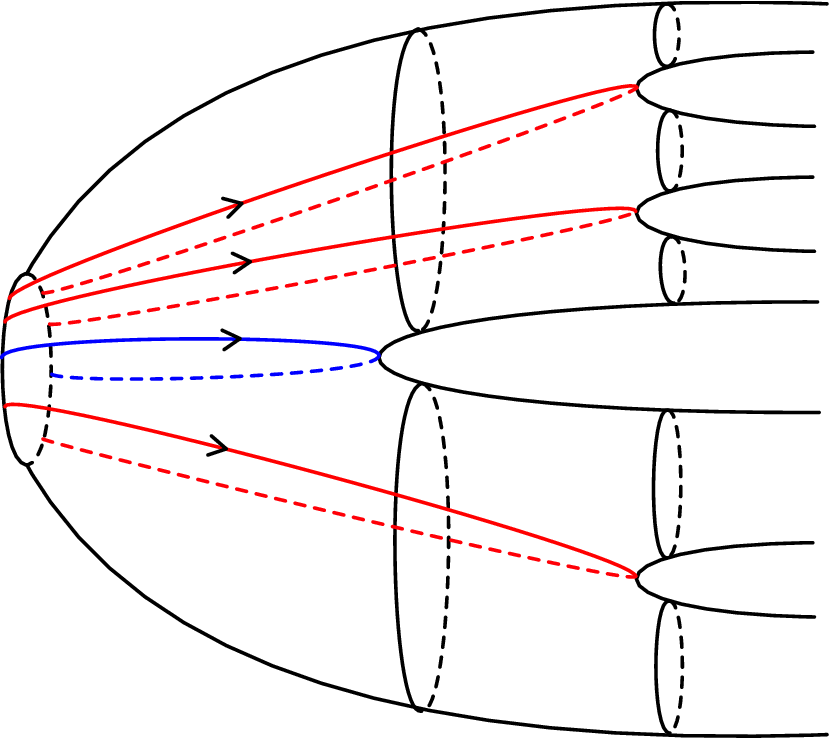}
\caption{The arcs in $\mathcal{C}_2$ are in red color, and an arc in $\Gamma_2^{(1)}$ is in blue color.}
\label{fig:C2}
\end{figure}
\end{enumerate}		
	
\item For $k \ge 3$, let  $\Gamma_{k-1} = \Gamma_{k-1}^{(1)} \cup \Gamma_{k-1}^{(2)}$ be the set of curves (as defined above) for the subsurface $S_{k-1}$. We know that the subsurface $S_k$ has $b+b_k$ boundary components.
\begin{enumerate}
\item  Since $S_{k-1} \subset S_k$, we choose a curve diagram $\Gamma_k^{(1)}$ for $\widehat{S}_k$ such that $\Gamma_{k-1}^{(1)} \subset \Gamma_k^{(1)}$ and  $\Gamma_k^{(1)}$ is totally disjoint from  $\Gamma_{k-1}^{(2)}$.

	\item We now define the set $\Gamma_k^{(2)}$ of $b_k-1$ separating ideal arcs on $S_k$ such that
\begin{itemize}
\item $\Gamma_k^{(2)}$  is totally disjoint from $\Gamma_k^{(1)}$,
\item $\Gamma_{k-1}^{(2)} \subset \Gamma_k^{(2)}$,
\item Each connected component of $S \setminus \Gamma_k^{(2)}$ contains precisely one connected component of $S \setminus S_k$. 
\end{itemize}
Note that, $S_k\setminus \Gamma_k^{(1)}$ is a 2-sphere with $1+b_k$ boundary components. Let $q_k: S_k\setminus \Gamma_k^{(1)} \to S_k$ be the quotient map. Observe that $S_k\setminus (\Gamma_k^{(1)}\cup \Gamma_{k-1}^{(2)})$ is a disjoint union of $b_{k-1}$ many 2-spheres with non-empty boundary. Again, repeating the procedure of step (2) for the case of $S_{k-1}$ for each such 2-sphere with non-empty boundary, we obtain a set $\mathcal{C}_k$ of ideal arcs  on $S_k\setminus (\Gamma_k^{(1)}\cup \Gamma_{k-1}^{(2)})$ such that their end points lie on $q_k^{-1}(\partial(S) \setminus (\Gamma_k^{(1)} \cup \Gamma_{k-1}))$. We define  $\Gamma_k^{(2)}= q_k(\mathcal{C}_k) \cup \Gamma_{k-1}^{(2)}$ and $\Gamma_k= \Gamma_k^{(1)}\cup \Gamma_k^{(2)}$.
\end{enumerate}		
\end{itemize}

Finally, we define 
\begin{equation}\label{explicit stable Alexander system gamma}
\Gamma = \cup_{k = 1}^\infty \Gamma_k.
\end{equation}
 We claim that $\Gamma$ is an ideal arc system for the surface $S$. Note that any two ideal arcs in $\Gamma$ are non-isotopic and disjoint. Thus, it remains to show that $\Gamma$ is a stable Alexander system for $S$.  Recall that, for each $k \ge 1$, $B_k$ is the set of boundary components of $S_k$ that are not the boundary components of $S$. Let us set $B= \cup_{k=1}^\infty B_k$. Then we have the following result.

\begin{lemma}\label{lemma:isobdy}
Let $S$ be an infinite-type surface with non-empty boundary and $f \in \Homeo^{+}(S,\partial S)$. If $f(\gamma)$ is isotopic to $\gamma$ for every $\gamma \in \Gamma$, then $f(b)$ is isotopic to $b$ for every $b \in B$.
\end{lemma}

\begin{proof}
Let $b \in B= \cup_{k=1}^\infty B_k$. Then $b$ is a boundary curve for some finite-type subsurface $S_i$ of $S$ as constructed in Proposition~\ref{prop:cmpexh}. We are given that 
$f(\gamma)$ is isotopic to $\gamma$ for every $\gamma \in \Gamma_i$. Since $\Gamma_i$ is finite, by Lemma~\ref{lemma:finalex}, there exists $h \in \operatorname{Homeo}^{+}(S,\partial S)$ isotopic to the identity on $S$ relative to $\partial(S)$ such that $h|_{\Gamma_{i}} = f|_{\Gamma_{i}}$. Taking $g = fh^{-1}$, we see that $g$ is isotopic to $f$ on $S$ relative to $\partial(S)$ and $g|_{\Gamma_{i}} = \id$. This implies that $g$ induces a map on $S \setminus \Gamma_{i}$. It follows from the construction of $\Gamma_i$ that the curve $b$ is isotopic  to a boundary component of one of the components of $S\setminus \Gamma_{i}$. This shows that $g(b)$ is isotopic to $b$ on $S \setminus \Gamma_{i}$. But, $g$ being isotopic to $f$ implies that $f(b)$ is isotopic to $b$ on $S$.
\end{proof} 

We need the following result \cite[Lemma 3.5]{MR4029627} for proving the next proposition.

\begin{lemma}\label{isotopy lemma}
Let $S$ be an oriented infinite-type surface, $\{S_k\}_{k \ge 1}$ be an exhaustion for $S$ and $B= \cup_{k=1}^\infty B_k$ as defined above. Let $f \in \Homeo^+(S, \partial S)$ be such that $f(b)$ is isotopic to $b$ for every $b \in B$. Then $f$ is isotopic to a homeomorphism $g \in  \Homeo^+(S, \partial S)$ such that $g|_B = \id$.
\end{lemma}

\begin{proposition}
The set $\Gamma$ is a stable Alexander system for the infinite-type surface $S$.
\end{proposition}

\begin{proof}
Let $f \in \Homeo^{+}(S,\partial S)$ be such that $f(\gamma)$ is isotopic to $\gamma$ for every $\gamma \in \Gamma$. By Lemma \ref{lemma:isobdy}, $f(b)$ is isotopic to $b$ for all $b \in B$. Further, by Lemma \ref{isotopy lemma}, upto isotopy, we can assume that $f|_b=\id|_b$ for all $b \in B$. Thus, the restriction $f|_{S_k}$ is a homeomorphism of $S_k$ for each $k \ge 1$.  Note that $S_k \cap (\cup_{\gamma\in \Gamma_{k+1}} \gamma)$ is a finite stable Alexander system for $S_k$ for each $k \ge 1$ (by condition (5) of Proposition \ref{prop:cmpexh}). By definition of a stable Alexander system, $f|_{S_k}$ is isotopic to the identity map on $S_k$ relative to $\partial S_k$. Suppose that $f$ is not isotopic to the identity map on $S$. Then there exists an essential simple closed curve (or an arc) $c$ such that $f(c)$ is not isotopic to $c$. Choose $k$ to be sufficiently large such that $c$ lies in $S_k$. This means that $f|_{S_k}(c)=f(c)$ is not isotopic to $c$, which is a contradiction. Hence, $f$ must be isotopic to the identity map on $S$.
\end{proof}

We are now in a position to prove our main result. To proceed, we first label our ideal arc system $\Gamma$ as follows:
\begin{itemize}
	\item Label the finite set of arcs $\Gamma_1$ in $S_1$ in any order, say $\gamma_1, \gamma_2, \dots , \gamma_{k_{1}}$.
	\item Continue the labelling for the finite set of arcs $\Gamma_2 \setminus \Gamma_1$ for $S_2$ as $\gamma_{k_1 +1}, \gamma_{k_1 +2}, \dots , \gamma_{k_2}$.
	\item Continue this process for each subsurface $S_i$.
\end{itemize}

We note that if $f, f' \in \Homeo^{+}(S,\partial S)$ represent the same mapping class in $\Mod(S)$, then after reducing $f(\Gamma)$ and  $f'(\Gamma)$ with respect to each other, we can assume that they are identical.

\begin{theorem}\label{main theorem}
Let $S$ be a connected oriented infinite-type surface with non-empty boundary. Then the big mapping class group $\Mod(S)$ is left-orderable.
\end{theorem}

\begin{proof}
Note that, for each $f \in \Mod(S)$,  $f(\Gamma)$ is also an ideal arc system for $S$. Further, for any $f,g \in \Mod(S)$, by Proposition \ref{prop:tight}, we can assume that $f(\Gamma)$ and $g(\Gamma)$ are reduced with respect to each other.
\par
 Let $f, g \in \Mod(S)$ such that $f \neq g$. Since $\Gamma$ is a stable Alexander system for $S$, then there exists $\gamma_l \in \Gamma$ such that $f(\gamma_l) \ne g(\gamma_l)$. Without loss of generality, we can take $l$ to be the minimum such index.  We define $f <_\Gamma g$ if $g(\gamma_l)$ branches off $f(\gamma_l)$ to the left (that is, if an initial segment of $g(\gamma_l)$ lies to the left of $f(\gamma_l)$). Note that, we are using the fixed orientation of the surface $S$ and that of the ideal arcs.
 \par 
Let $f,g,h \in \Mod(S)$ with $f <_\Gamma g$ and $g <_\Gamma h$. In view of Propositions~\ref{prop:tight} and \ref{prop:reduction}, we can assume that $f(\Gamma), g(\Gamma), h(\Gamma)$ are pairwise reduced. Let $\gamma_i, \gamma_j \in \Gamma$ be the minimal indexed arcs such that $f(\gamma_i)\neq g(\gamma_i)$ and $g(\gamma_j)\neq h(\gamma_j)$. Thus,  $g(\gamma_i)$ branches off $f(\gamma_i)$ to the left and $h(\gamma_j)$ branches off $g(\gamma_j)$ to the left. Therefore,  $h(\gamma_k)$ branches off $f(\gamma_k)$ to the left where $k =\min\{i,j\}$, and consequently $f <_\Gamma h$.
\par

If $f,g,h \in \Mod(S)$ such that $f <_\Gamma g$, then $hf <_\Gamma hg$, since the homeomorphism of $S$ representing $h$ applied to $f(\Gamma)$ and $g(\Gamma)$ leaves the ideal arc systems reduced with respect to each other. This shows that the order $<_\Gamma$ is left-invariant, and the proof is complete.
\end{proof}

It is clear from the proof that the ordering $ <_\Gamma$ depends heavily on the ideal arc system $\Gamma$.
\medskip

\section{Generalised ideal arc systems and comparison of orderings}\label{section comparison of orderings}
In this section, we introduce a generalised ideal arc system on a surface, and prove that it also induces a left-ordering on the mapping class group. Further, we examine the conditions under which two generalised ideal arc systems induce the same left-ordering. This is a step towards understanding the space of all left-orderings on the big mapping class group. The next definition is borrowed from \cite[Definition 4.1]{MR1805402}.

\begin{definition}
A {\it generalised ideal arc system} for a surface $S$ is a countable labelled set $\Gamma= \{\gamma_k\}_{k \ge 1}$ of non-isotopic arcs on $S$ satisfying the following conditions: 
\begin{enumerate}
\item  $\cup_{k \ge 1} \interior(\gamma_k)$ is an embedding into $S$ and is disjoint from $\partial S \cup P$.
\item The starting point of $\gamma_i$ lies on $\cup_{k=1}^{i-1} \gamma_k \cup \partial S$.
\item The end point of $\gamma_i$ lies on $\cup_{k=1}^{i-1} \gamma_k \cup \interior(\gamma_i) \cup \partial S \cup P$.
\item $\Gamma$ is a stable Alexander system for $S$. 
\end{enumerate}
\end{definition}

\begin{example}
The ideal arc system for $S$ as in \eqref{explicit stable Alexander system gamma} is a specific generalised ideal arc system.
\end{example}

Next, we describe how to establish an ordering on $\Mod(S)$ using a generalised ideal arc system. 

\begin{proposition}\label{gen ideal arc implies left-orderability}
Let $S$ be a connected oriented infinite-type surface with non-empty boundary and $\Gamma$ be a generalised ideal arc system for $S$. Then $\Gamma$ determines a left-ordering on $\Mod(S)$.
\end{proposition}

\begin{proof}
Let $f, g\in \Mod(S)$. Let $i$ be the smallest index such that $f(\gamma_i)$ and $g(\gamma_i)$ are not isotopic. By Lemma~\ref{lemma:finalex}, we can replace $f$ or $g$ by isotopic maps such that the restrictions of $f$ and $g$ to $\cup_{k=1}^{i-1} \gamma_k$ agree. 
\par
Suppose that the end point of $\gamma_i$ does not lie in $\interior(\gamma_i)$. Since the starting point of $\gamma_i$ lies on $\cup_{k=1}^{i-1} \gamma_k \cup \partial S$, and both $f,g$ agree on $\cup_{k=1}^{i-1} \gamma_k \cup \partial S$ as maps, it follows that $f(\gamma_i)$ and $g(\gamma_i)$ have the same starting point. Consequently, they both lie in the same path component of  $S \setminus f(\cup_{k=1}^{i-1} \gamma_k)$. Next, we isotope $f$ so as to minimise the number of intersections of $f(\gamma_i)$ and $g(\gamma_i)$ while keeping $f(\cup_{k=1}^{i-1} \gamma_k)$   fixed throughout the isotopy. Finally, we observe that the initial segments of $f(\gamma_i)$ and $g(\gamma_i)$ depart from their common starting point into the interior of $S \setminus f(\cup_{k=1}^{i-1} \gamma_k)$ in different directions, with one going more to the left than the other. We say that $f <_\Gamma g$ if $g(\gamma_i)$ branches off $f(\gamma_i)$ to the left; otherwise, we say that $g<_\Gamma f$. 
\par 
Suppose that the end point of $\gamma_i$ lie in $\interior(\gamma_i)$. In this case, we consider the arc $\gamma_i'$, which is obtained by sliding the end point of $\gamma_i$ backwards along $\gamma_i$  such that the start and end point coincide. See Figure \ref{pulling end-point}. We compare $f(\gamma_i')$ and $g(\gamma_i')$ as in the preceding case. It is routine to verify that the ordering $<_\Gamma$ is a left-ordering on $\Mod(S)$.
\end{proof}

\begin{figure}[H]
	\labellist
	\tiny
	\pinlabel $\gamma_i$ at 130, 123
	\pinlabel $\gamma_i'$ at 430, 123
	\endlabellist
\includegraphics[width=4in]{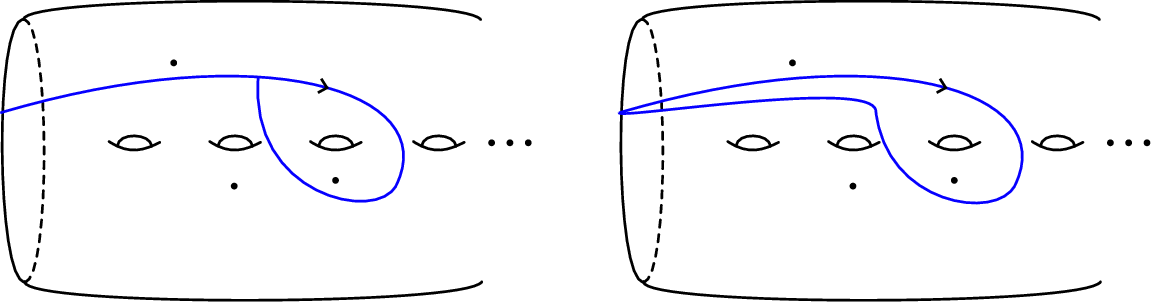}
\caption{Pulling the end point of $\gamma_i$ to the boundary of the surface $S$.}
\label{pulling end-point}
\end{figure} 

 Let $\mathcal{I}$ be the set of all generalised ideal arc systems for $S$. We can view a generalised ideal arc system as a subset of all maps from $\sqcup_{k \ge 1}~ I_k$ to $S$, where $I_k=[0,1]$. With this view, $\mathcal{I}$ can be equipped with the compact open topology. 
 
\begin{definition}
A {\it loose isotopy} between two generalised ideal arc systems on $S$ is generated by the following three types of equivalences:
\begin{enumerate}
\item Continuous deformation: Two generalised ideal arc systems $\Gamma$ and $\Delta$ are considered equivalent if they belong to the same path component of $\mathcal{I}$. Equivalently, there exists a continuous map $H:  (\sqcup_{k \ge 1}~ I_k) \times I \to S$ such that 
\begin{enumerate}
\item $H|_{I_k\times \{0\}}=\gamma_k$ and $H|_{I_k\times \{1\}}=\delta_k$ for each $k \ge 1$.
\item $\{H|_{I_k \times t}\}_{k \ge 1}$ is a generalised ideal arc system for each $t \in I$.
\end{enumerate}
\begin{figure}[H]
		\labellist
	\tiny
	\pinlabel $\gamma_1$ at 130, 123
	\pinlabel $\gamma_1'$ at 150, 90
	\pinlabel $\gamma_2$ at 340, 115
	\pinlabel $\gamma_2'$ at 400, 87
	\pinlabel $\gamma_1$ at 330, 85
	\endlabellist
	\includegraphics[width=5in]{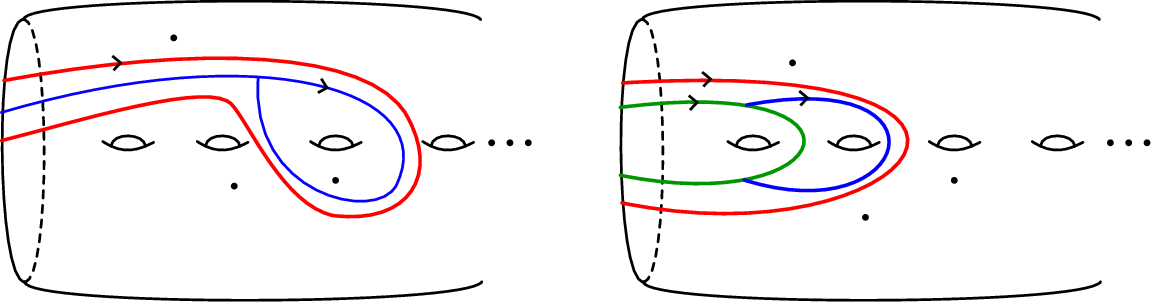}
	\caption{Continuous deformation of $\gamma_1$ to $\gamma_1'$ in the first subfigure and continuous deformation of $\gamma_2$ to $\gamma_2'$ in the second subfigure.}
	\label{continuous_deformation}
\end{figure}
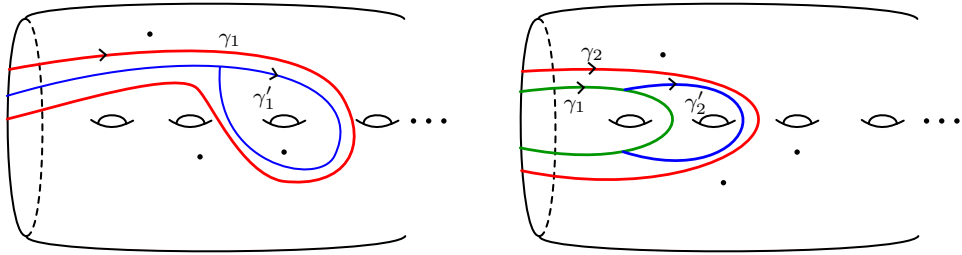
\item Pulling loops tight around punctures: If a segment of an arc $\gamma_i$ bounds a disk with one puncture, then it can be pulled tight so that the end point of $\gamma_i$ is the puncture itself. See Figure \ref{pulling around puncture}.
\begin{figure}[H]
\includegraphics[width=5in]{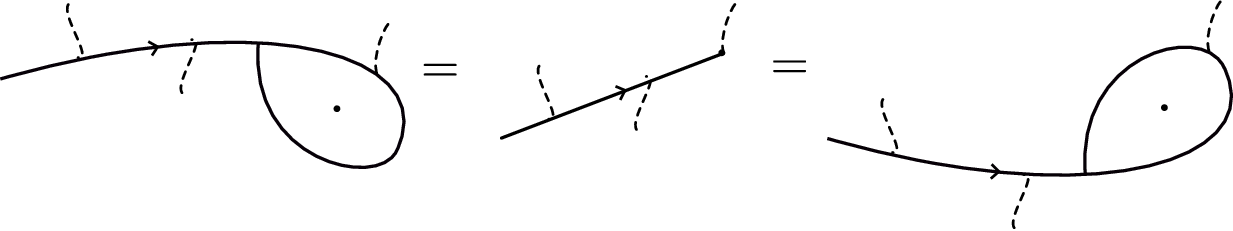}
\caption{Pulling a loop tight around a puncture.}
\label{pulling around puncture}
\end{figure}
\item If for some $i$, one of the  components of $S \setminus \cup_{k=1}^{i-1} \gamma_k$ is a torus with one boundary component or an annulus, and the arc $\gamma_i$ lies in that connected component, then its orientation can be reversed.
\end{enumerate}
 \end{definition}

 \begin{remark}\label{deforming to puncture looping arc}
 For each arc $\gamma_i$ of $\Gamma= \{\gamma_k \}_{k \ge 1}$, let $C_i$ denote the component of $S\setminus \cup_{k=1}^{i-1} \gamma_k$ that contains $\gamma_i$. If $C_i \setminus \gamma_i$ has a  component that is a disk with one puncture, then we say that  $\gamma_i$ is an {\it almost puncture looping arc}. For each such arc $\gamma_i$, by continuous deformation and pulling tight around a puncture, we see that $\Gamma$ is loosely isotopic to a generalised ideal arc system $\Gamma'=  \{\gamma'_k \}_{k \ge 1}$ (see Figure \ref{fig:ias1}) such that
 \begin{itemize}
\item  $\gamma_k=\gamma'_k$ for all $1 \le k  \le i-1$,
\item $\gamma'_i$ is obtained from $\gamma_i$ by first sliding its end point along $\cup_{k=1}^{i-1} \gamma_k \cup \partial S$ and then pulling it tight so that its end point is the prescribed puncture.
  \end{itemize}
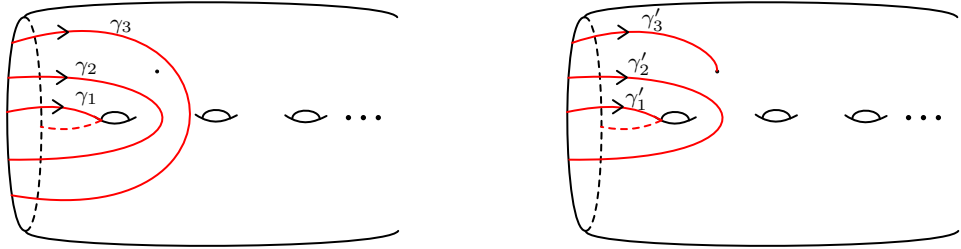
\begin{figure}[H]
	\labellist
	\tiny
	\pinlabel $\gamma_1$ at 35, 65
	\pinlabel $\gamma_3$ at 50, 96
	\pinlabel $\gamma_2$ at 35, 78
	\pinlabel $\gamma_1'$ at 273, 65
	\pinlabel $\gamma_3'$ at 280, 99
	\pinlabel $\gamma_2'$ at 274, 80
	\endlabellist
	\includegraphics[width=5in]{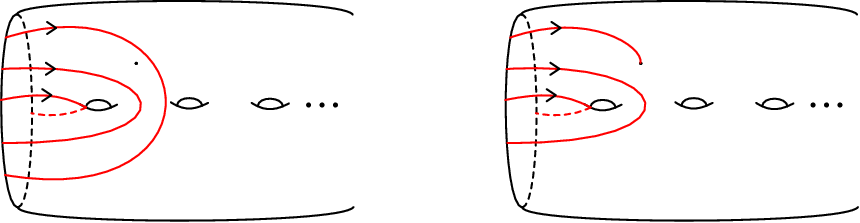}
	\caption{Loose isotopy between $\Gamma$ and $\Gamma'$.}
	\label{fig:ias1}
\end{figure}
\end{remark}

It is known that the isotopy classes of essential ideal arcs on the torus with one boundary component are in one-to-one correspondence with the isotopy classes of essential
simple closed curves on the torus (see \cite{MR4584864}). Thus, we may refer to  simple closed curves on the torus or equivalently ideal arcs on the torus with one boundary component by their associated slopes, which are of the form $p/q$ for some co-prime integers $p$ and $q$. We note that an ideal arc with slope $p/q$ intersects the meridian $|p|$ times and the longitude $|q|$ times in minimal position. For convenience, we refer to an ideal arc with slope $p/q$ on the torus with one boundary component as a {\it $(p,q)$ ideal arc}.

\begin{lemma}\label{lemma:S11}
Let $\Gamma = \{\gamma_1, \gamma_2\}$ be an ideal arc system for the torus $S_1^1$  with one boundary component and $\Delta=\{\delta_1, \delta_2\}$ be another ideal arc system obtained from $\Gamma$ by reversing the orientation of one or both of the arcs. Then both the ideal arc systems induce the same left-ordering on $\Mod(S_1^1)$.
\end{lemma}

\begin{proof}
 By the  change of coordinate principle, if two ideal arc systems of $S_1^1$ have the same set of end points, then there exists a $h \in \Mod(S_1^1)$ which maps one ideal arc system onto the other. Thus, it is sufficient to prove the assertion for the ideal arc system $\Gamma = \{\gamma_1, \gamma_2\}$ as shown in Figure~\ref{fig:ias}.
\begin{figure}[H]
	\labellist
	\tiny
	\pinlabel $\gamma_1$ at 75, 112
	\pinlabel $\gamma_2$ at 50, 93
		\endlabellist
\includegraphics[width=1.8in]{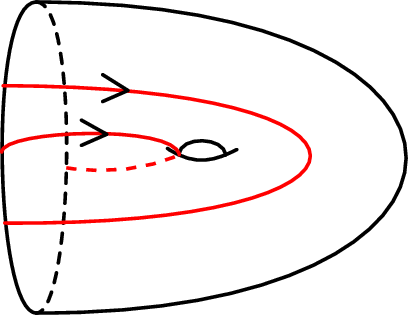}
\caption{An ideal arc system for the torus with one boundary component.}
\label{fig:ias}
\end{figure}

If $f \in \Mod(S_1^1)$, then either $f(\gamma_1)$ is an ideal arc of the form $(\pm p,\pm q)$, where $p, q \in \mathbb{N}$ with $\gcd(p,q)= 1$, or an ideal arc of the form $(\pm 1, 0)$ or $(0, \pm 1)$.
\begin{itemize}
\item  If $f(\gamma_1)$ is a $(p,q)$ or $(-p,-q)$ or $(-1,0)$ or $(0,\pm 1)$ ideal arc, then $\gamma_1$ branches off $f(\gamma_1)$ to the left.
Similarly,  $\delta_1$  branches off  $f(\delta_1)$ to the left. Consequently, $f$ is a negative element in both the orderings $<_\Gamma$ and $<_\Delta$. Figures \ref{fig:23} and \ref{fig:-10} illustrate the cases $(p,q) = (3, 2)$, $(-1, 0)$ and $(0,1)$.

\begin{multicols}{2}
\begin{figure}[H]
	\labellist
	\tiny
	\pinlabel $\gamma_1$ at 125, 135
	\pinlabel $f(\gamma_1)$ at 100, 98
		\endlabellist
\includegraphics[width=1.8in]{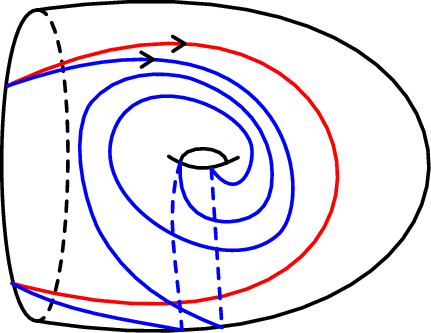}
\caption{$f(\gamma_1) = (3, 2)$.}
\label{fig:23}
\end{figure}
\begin{figure}[H]
	\labellist
	\tiny
	\pinlabel $\gamma_1$ at 145, 130
	\pinlabel $(0,1)$ at 55, 98
	\pinlabel $(-1,0)$ at 55, 72
		\endlabellist
\includegraphics[width=1.8in]{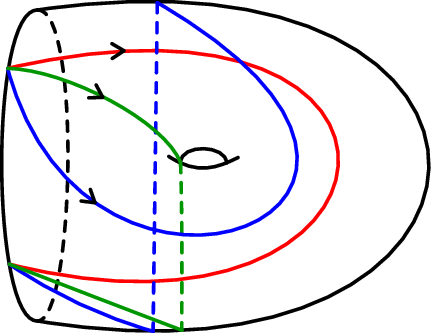}
\caption{$f(\gamma_1) = (-1,0)$ in blue and  $f(\gamma_1)=(0,1)$ in green.}
\label{fig:-10}
\end{figure}
\end{multicols}
\item If $f(\gamma_1)$ is a $(-p,q)$ or $(p,-q)$ ideal arc, then $f(\gamma_1)$ branches off $\gamma_1$ to the left.  Similarly, $f(\delta_1)$ branches off $\delta_1$ to the left, and hence $f$ is a positive element in both the orderings $<_\Gamma$ and $<_\Delta$.  Figure~\ref{fig:-23} illustrates the case $(p,q) = (3, -2)$.
\begin{figure}[H]
	\labellist
	\tiny
	\pinlabel $\gamma_1$ at 125, 135
	\pinlabel $f(\gamma_1)$ at 100, 96
		\endlabellist
\includegraphics[width=1.8in]{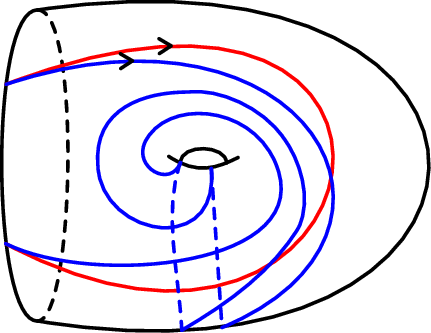}
\caption{$f(\gamma_1) = (3, -2)$.}
\label{fig:-23}
\end{figure}

\item If $f(\gamma_1)=\gamma_1$ with the same orientation, that is, $f(\gamma_1)$ is a $(1,0)$ ideal arc, then $f$ descends as a mapping class of the annulus $S \setminus \gamma_1$. In this case, $f$ is some power of the Dehn twist $T_c$ along the central curve $c$ of the annulus $S \setminus \gamma_1$. If both $\gamma_2$ and $\delta_2$ descend to $S \setminus \gamma_1$ with the same orientation, then there is nothing to prove. Suppose that $\gamma_2$ and $\delta_2$ descend to $S \setminus \gamma_1$ with opposite orientations.  If $f= T_c^n$ for some $n>0$, then $ f(\gamma_2)$ branches off $\gamma_2$ to the left, and $f(\delta_2)$ branches off $\delta_2$ to the left. If $f= T_c^n$ for some $n<0$, then $ \gamma_2$ branches off $f(\gamma_2)$ to the left, and $\delta_2$ branches off $f(\delta_2)$ to the left.  Thus, in each case, $f$ is a positive element with respect to $<_\Gamma$ if and only if it is a positive element with respect to $<_\Delta$.  See Figure~\ref{fig:T_c(gamma_2)} for an illustration.
\begin{figure}[H]
	\labellist
	\tiny
	\pinlabel $c$ at 180, 117
	\pinlabel $\gamma_2$ at 50, 80
	\pinlabel $T_c(\gamma_2)$ at 225, 50
		\endlabellist
\includegraphics[width=2in]{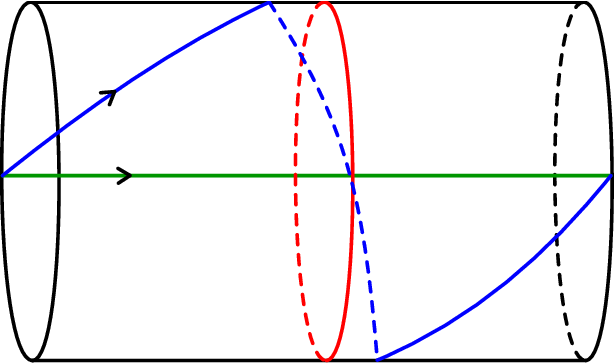}
\caption{The curve $T_c(\gamma_2)$.}
\label{fig:T_c(gamma_2)}
\end{figure}
\end{itemize}
Thus, in all the cases, the orderings induced by $\Gamma$ and $\Delta$ are the same.
\end{proof}

The proof of the following lemmas are along similar lines as that of \cite[Theorem 5.2(a)]{MR1805402}, but we present them here for the sake of completeness.

\begin{lemma}\label{lem:pullingtigth}
Let $\Delta$ be an ideal arc system for $S$ such that an ideal arc $\delta_i \in \Delta$ is a loop enclosing exactly one puncture, and let $\Gamma$ be another ideal arc system obtained from $\Delta$ by squashing the arc $\delta_i$ into an arc $\gamma_i$ that has the same starting point as that of $\delta_i$ and has the puncture as its end point.  Then $\Delta$ and $\Gamma$ induce the same left-ordering on $\Mod(S)$.
\end{lemma}

\begin{proof}
 Let $f, g \in \Mod(S)$  be such that $f <_\Gamma g$. We claim that $f <_{\Delta} g$. 
 \begin{itemize}
 \item   If $f(\gamma_j) \neq g(\gamma_j)$ for some $1 \le j \le i-1$, then the claim is evident since the first $i-1$ arcs of $\Gamma$ and $\Delta$ are the same. 
 \item If $f(\gamma_j)=g(\gamma_j)$ for each $1 \le j \le i$ and $f(\gamma_k) \neq g(\gamma_k)$ for some $i<k$, then, obviously, the first $i-1$ arcs of $f(\Delta)$ and $g(\Delta)$ are the same. Since the boundary curves of sufficiently small regular neighbourhoods of $f(\gamma_i)$ and $g(\gamma_i)$ are isotopic to $f(\delta_i)$ and $g(\delta_i)$ respectively, it follows that $f(\delta_i)$ and $g(\delta_i)$ are also isotopic. Hence, the claim follows in this case as well.
 \item Finally, suppose that $f(\gamma_j)=g(\gamma_j)$  for each $1 \le j \le i-1$ and $f(\gamma_i) \ne g(\gamma_i)$. The arcs $f(\gamma_i)$ and $g(\gamma_i)$ are reduced with respect to each other such that $g(\gamma_i)$ branches off $f(\gamma_i)$ to the left. Since the boundary curves of sufficiently small regular neighbourhoods of $f(\gamma_i)$ and $g(\gamma_i)$ are isotopic to $f(\delta_i)$ and $g(\delta_i)$, respectively, it follows that  $g(\delta_i)$ branches off $f(\delta_i)$ to the left, and hence  $f <_{\Delta}g$ (see Figure~\ref{fig:nbd}). 
 \end{itemize}
\end{proof} 
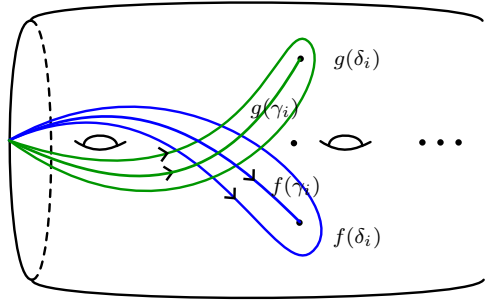
\begin{figure}[H]
		\labellist
	\tiny
	\pinlabel $f(\gamma_i)$ at 140, 55
	\pinlabel $g(\gamma_i)$ at 130, 93
	\pinlabel $g(\delta_i)$ at 170, 117
    \pinlabel $f(\delta_i)$ at 170, 30
	\endlabellist
\includegraphics[width=2.5in]{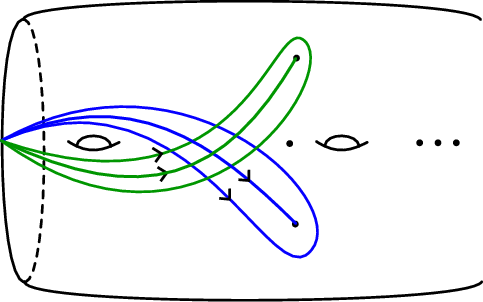}
\caption{Small regular neighbourhoods of arcs $f(\gamma_i)$ and $g(\gamma_i)$ isotopic to $f(\delta_i)$ and $g(\delta_i)$, respectively.} 
\label{fig:nbd}
\end{figure} 

\begin{lemma}\label{lem:condef}
If $\Gamma$ and $\Delta$ are ideal arc systems for $S$ such that both are continuously deformable, then they induce the same left-ordering on $\Mod(S)$.
\end{lemma}

\begin{proof}
We denote an element of $\Gamma$ by $\gamma_i$ and the corresponding element of $\Delta$ by $\delta_i$. Let $f \in \Mod(S)$ and let $i$ be the smallest index such that $f(\gamma_i) \ne \gamma_i$. Since $\gamma_k$ and $\delta_k$ are continuously deformable for each $1 \le k \le i-1$, it follows that $f(\delta_k)= \delta_k$ for each $1 \le k \le i-1$. Let $H: I \times I_i \to S$ be the continuous deformation with $H(\{0\}\times I_i) = \gamma_i$ and $H(\{1\}\times I_i) = \delta_i$. We reduce $\gamma_i$ and $f(\gamma_i)$ if necessary. If $f(\gamma_i)$ branches off $\gamma_i$ to the left, then after applying $H$, the arc $f(H(\{t\}\times I_i))$ branches off $H(\{t\}\times I_i)$ to the left for each $t\in I$. Hence, $f(\delta_i)$ branches off $\delta_i$ to the left. Interchanging the roles of $\Gamma$ and $\Delta$ implies that the left-orderings $<_\Gamma$ and $<_\Delta$ are the same.
\end{proof}

Lemmas  \ref{lemma:S11}, \ref{lem:pullingtigth} and \ref{lem:condef} lead to the following result.

\begin{proposition}	\label{prop:loose1}
Any two loosely isotopic  generalised ideal arc systems for $S$ determine the same left-ordering on $\Mod(S)$.
\end{proposition}

\begin{lemma}\label{lem:disjointcurves}
Let $S$ be a connected oriented surface of finite or infinite-type with non-empty boundary. Let $\gamma$ be an ideal arc on $S$ that starts and ends on a boundary component of $S$ such that $S\setminus \gamma$ has no component homeomorphic to a disk, a punctured disk, or an annulus. Let $\delta$ be an ideal arc on $S$ that is isotopic to $\gamma$, but equipped with opposite orientation. Then, there exist non-isotopic disjoint simple closed curves $c_1$ and $c_2$ on $S$ such that, with respect to the orientation, the first non-trivial intersection of $\gamma$ is with $c_1$ and the first non-trivial intersection of $\delta$ is with $c_2$.
\end{lemma}

\begin{proof}
If the start and the end point of $\gamma$ lie on different components of $\partial S$, then we choose $c_1$ and $c_2$ to be simple closed curves homotopic to the boundary component containing the start and the end point, respectively. Now, we assume that the starting and the end point of $\gamma$ (and of $\delta$) lie on the same component of $\partial S$. If $\gamma$ is non-separating, then by the change of coordinate principle, there exists $h \in \Homeo^+(S, \partial S)$ which takes the arc $\gamma$ to an arc $a$ as shown in Figure~\ref{fig:nonsep}. We can choose non-isotopic disjoint simple closed curves $d_1$ and $d_2$ satisfying the assertion of the lemma for the arc $a$ (see Figure~\ref{fig:nonsep}). Then $c_1 := h^{-1}(d_1)$ and $c_2:= h^{-1}(d_2)$ are the desired simple closed curves for $\gamma$ and $\delta$.

	\begin{figure}[H]
		\labellist
		\tiny
		\pinlabel $a$ at 50, 58
		\pinlabel $d_1$ at 110, 150
		\pinlabel $d_2$ at 110, 50
		\endlabellist
		\includegraphics[width=2in]{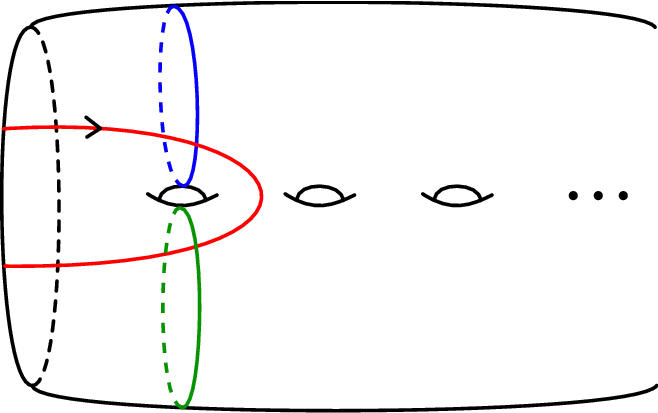}
		\caption{Curves $d_1$ and $d_2$ when the arc $a$ is non-separating.}
\label{fig:nonsep}
	\end{figure}
If $\gamma$ is separating, then we write $S\setminus \gamma = S'\cup S''$. We consider two cases as follows:
\par

Case 1: Suppose that each of $S'$ and $S''$ is either a torus with one boundary component or a disk with two punctures. In this case, the simple closed curves $c_1$ and $c_2$ as shown in Figure~\ref{fig:sep11} are the desired non-isotopic disjoint curves.

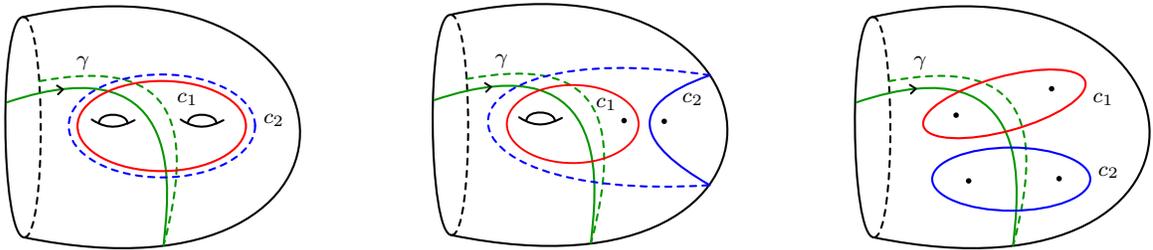
\begin{figure}[H] 
	\labellist
	\tiny
	\pinlabel $\gamma$ at 30, 72
	\pinlabel $c_1$ at 70, 58
	\pinlabel $c_2$ at 103, 50
	\pinlabel $\gamma$ at 190, 73
	\pinlabel $c_1$ at 230, 56
	\pinlabel $c_2$ at 263, 58
	\pinlabel $\gamma$ at 350, 72
	\pinlabel $c_1$ at 420, 58
	\pinlabel $c_2$ at 422, 30
	\endlabellist
	\includegraphics[width=6in]{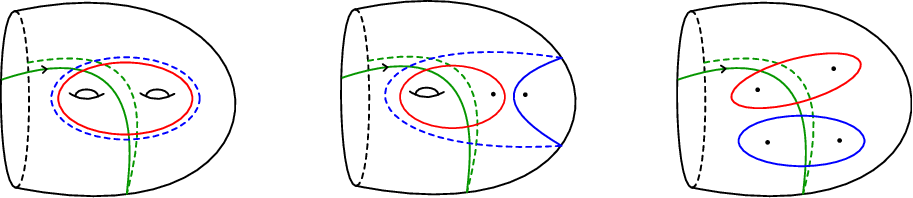}
	\caption{The curves $c_1$ is in red color, $c_2$ is in blue color and the arc $\gamma$ is in green color.}
	\label{fig:sep11}
\end{figure}
\par
Case 2: Suppose that at least one of $S'$ or $S''$ is neither a disk with two punctures nor a torus with one boundary component. Thus, at least one of them, say $S'$, is one of the following:
\begin{itemize}
\item A disk with at least three puctures.
\item A torus with at least two boundary components.
\item A torus with at one boundary component and at least one puncture.
\item A surface with genus at least two and at least one boundary component.
\end{itemize}
In this case, there exist non-isotopic disjoint simple closed curves $c_1$ and $c_2$ with desired properties, for example as shown in Figure~\ref{fig:sep12}.
	\begin{figure}[H] 
	\labellist
	\tiny
	\pinlabel $S'$ at 50, 92
	\pinlabel $S'$ at 320, 92
	\pinlabel $\gamma$ at 30, 80
	\pinlabel $c_1$ at 80, 66
	\pinlabel $c_2$ at 110, 48
	\pinlabel $\gamma$ at 265, 80
	\pinlabel $c_1$ at 320, 66
	\pinlabel $c_2$ at 360, 48
	\endlabellist
	\includegraphics[width=5in]{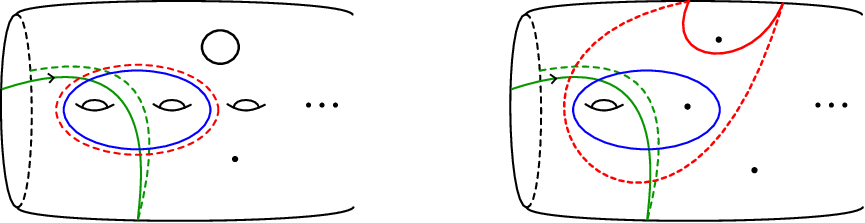}
	\caption{The curves $c_1$ is in blue color, $c_2$ is in red color and the arc $\gamma$ is in green color.}
	\label{fig:sep12}
\end{figure}
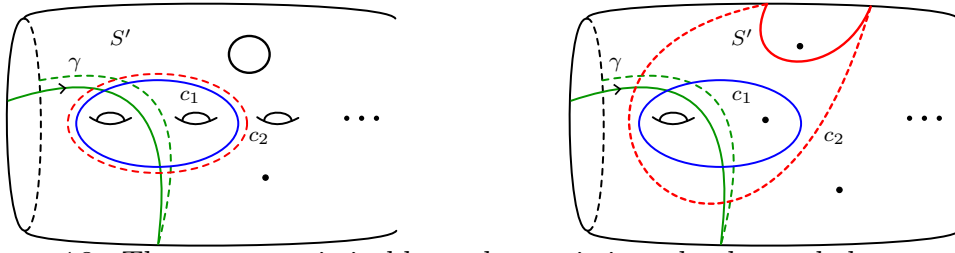
 This complete the proof of the lemma.
\end{proof}

\begin{theorem}\label{thm:loose2}
Let $S$ be a connected oriented surface with non-empty boundary. Two generalised ideal arc systems for $S$ determine the same left-ordering on $\Mod(S)$ if and only if they are loosely isotopic.
\end{theorem}

\begin{proof}
In view of Proposition~\ref{prop:loose1}, it remains to prove that if $\Gamma$ and $\Delta$ are two generalised ideal arc systems for $S$ which determine the same left-ordering on $\Mod(S)$, then they are loosely isotopic. In view of Remark \ref{deforming to puncture looping arc}, upto loose isotopy, we can assume that none of $\Gamma$ and $\Delta$ have any almost puncture looping arc. Applying continuous deformation, if needed, we can further assume that the starting point of each arc in  $\Gamma$ as well as $\Delta$ is on $\partial S$ and the end point of each arc is on $\partial S \cup P$.  Furthermore, we can assume, up to continuous deformation, that $\Gamma$ and $\Delta$ are reduced with respect to each other. Indeed, if there is a bigon enclosed by a pair of segments of arcs or a triangle whose one side is a segment of the boundary and the other two sides are segments of a pair of arcs from $\Gamma \cup \Delta$, then we can remove those bigons and triangles. Here, by removing a triangle, we mean sliding one segment of a pair of arcs from $\Gamma \cup \Delta$ onto the other along a segment of the boundary.
\par 

We claim that, for each $n \ge 2$, there exists an isotopy $H_n$  (which need not fix $\partial S$ point-wise) between $\cup_{j=1}^n \gamma_j$ and $\cup_{j=1}^n \delta_j$ such that $H_n$ restricted to the first $n-1$ arcs is the isotopy $H_{n-1}$. We  construct such an isotopy inductively.   By our assumption on $\Gamma$ and $\Delta$, the end points of the arcs $\gamma_1$ and  $\delta_1$  lie either on $\partial S$ or in $P$. If $\gamma_1$ and $\delta_1$ are not isotopic, then there are two possible cases:
\par	

Case 1: Suppose that $\gamma_1$ and $\delta_1$ are isotopic with orientation of one of them reversed.  In this case, the arcs $\gamma_1$ and $\delta_1$ have end points on $\partial S$. By Lemma~\ref{lem:disjointcurves}, there exist disjoint simple closed curves $c_1$ and $c_2$ such that the first non-trivial intersection of $\gamma_1$ is with $c_1$ and that of $\delta_1$ is with $c_2$. Considering the arcs $T_{c_2}^{-1} T_{c_1}(\gamma_1)$ and $T_{c_2}^{-1} T_{c_1}(\delta_1)$ and reducing them with respect to $\gamma_1$ and $\delta_1$ if necessary, we see that $\id <_\Gamma T_{c_2}^{-1} T_{c_1}$ and $\T_{c_2}^{-1} T_{c_1} <_\Delta \id$. \par	

Case 2: Suppose that $\gamma_1$ and $\delta_1$ are not isotopic even with orientation of any of them reversed. In this case, there exist essential simple closed curves $\tau, \tau'$ such that $$i(\tau, \gamma_1) = 0,~~ i(\tau, \delta_1) \neq 0,~~ i(\tau', \gamma_1) \neq 0,~~ \text{ and }~ i(\tau', \delta_1) = 0.$$
Indeed, if $\gamma_1$ and $\delta_1$ intersect, then choose $\tau$ to be one of the boundary curves of a regular neighbourhood $\partial S \cup \gamma_1$ that intersects $\delta_1$. Similarly, choose $\tau'$  to be one of the boundary curves of a regular neighbourhood of $\partial S \cup \delta_1$ that intersects $\gamma_1$. If $\gamma_1$ and $\delta_1$ do not intersect, then choose $\tau$ to be an essential simple closed curve in $S\setminus \gamma_1$ that intersects $\delta_1$, and  choose $\tau'$ to be an essential simple closed curve in $S\setminus \delta_1$ that intersects $\gamma_1$. Considering the arcs $T_{\tau'}^{-1} T_{\tau}(\gamma_1)$ and $T_{\tau'}^{-1} T_{\tau}(\delta_1)$ and reducing them with respect to $\gamma_1$ and $\delta_1$ if necessary, we obtain $T_{\tau'}^{-1} T_{\tau} <_\Gamma \id$ and $\id <_\Delta T_{\tau'}^{-1} T_{\tau}$.
\par

Thus, both the cases contradict the hypothesis that $\Gamma$ and $\Delta$ induce the same left-ordering. This shows that $\gamma_1$ and $\delta_1$ must be isotopic. We fix an isotopy, say $H_1$, between $\gamma_1$ and $\delta_1$.
\par

We denote the image of $H_1$ by $H_1$ itself. By our assumption on $\Gamma$ and $\Delta$, the end points of both the arcs $\gamma_2$ and $\delta_2$ lie either on $\partial S$ or in $P$. For any $j\geq 2$, if the arcs $\gamma_j$ or $\delta_j$ intersected $H_1$, then the intersection creates a bigon or a triangle containing segments of the arcs $\gamma_1$ or $\delta_1$, which is impossible due to our assumption that $\Gamma \cup \Delta$ is reduced. Hence, $H_1$  does not intersect the arcs $\gamma_j$ and $\delta_j$ for each $j\geq 2$. Now, we consider the surface $S \setminus H_1$.  If $\gamma_2$ and $\delta_2$ lie in different connected components of $S \setminus H_1$, then $\gamma_2$  and $\delta_2$ are not isotopic on $S$. Using  arguments similar to the Case 2 above, we get a contradiction to fact that $\Gamma$ and $\Delta$ induce the same left-ordering. Hence, $\gamma_2$ and $\delta_2$ lie in  the same connected component, say $C_1$, of $S \setminus H_1$. By repeating the procedure that we applied for  $\gamma_1$ and $\delta_1$ on $S$,  we obtain an isotopy between $\gamma_2$ and $\delta_2$ on $C_1$. We extend this isotopy to an isotopy $H_2$ between $\gamma_1 \cup \gamma_2$ and $\delta_1 \cup \delta_2$, which restricts to the isotopy $H_1$ between $\gamma_1$ and $\delta_1$. We note that, if $C_1$ is a torus with one boundary component or an annulus, then the situation of Case 1  (that is, $\gamma_2$ and $\delta_2$ are isotopic with orientation of one of them reversed) induces the same left-ordering, and hence they are loosely isotopic due to Lemma \ref{lemma:S11}.
\par
	
For $n \ge 2$, let us assume that there exists an isotopy $H_{n-1}$ between $\cup_{j=1}^{n-1} \gamma_j$ and $\cup_{j=1}^{n-1} \delta_j$. Again, by our assumption on $\Gamma$ and $\Delta$, the arcs $\gamma_n$ and $\delta_n$ are such that their end point are either on $\partial S$ or in $P$.	Similar to the case $n=2$, the isotopy $H_{n-1}$ does not intersect the arcs $\gamma_j$ and $\delta_j$ for each $j\geq n$, and both $\gamma_n$ and $\delta_n$ lie in the same connected component $C_{n-1}$ of $S \setminus H_{n-1}$.   By repeating the procedure that we applied for  $\gamma_2$ and $\delta_2$ on $S$,  we obtain an isotopy between $\gamma_n$ and $\delta_n$  on  $C_{n-1}$. We then  extend this isotopy to an isotopy $H_n$ between $\cup_{j=1}^{n} \gamma_j$ and $\cup_{j=1}^{n} \delta_j$, which restricts to the isotopy $H_{n-1}$ between  $\cup_{j=1}^{n-1} \gamma_j$ and $\cup_{j=1}^{n-1} \delta_j$. This proves our claim. 
\par 

We define $$H: \left(\sqcup_{k \ge 1} I_k \right) \times I \to S$$ such that its restriction to $(\sqcup_{k =1}^n I_k) \times I$ is the map $H_n$. Since $H_m$ restricted to $(\sqcup_{k =1}^n I_k) \times I$  is $H_n$ for each $m >n$, such an $H$ is well-defined and uniquely determined by $\{H_n\}_{n \ge 1}$. With this the proof of the theorem is complete.
\end{proof}

\begin{definition}
Two left-orderings $<$ and $<'$ on a group $G$ are said to be {\it conjugate} if there exists $z \in G$ such that $x <' y$ iff $ xz < yz$. In this case, we denote $<'$ by $<^z$.
    \end{definition}

\begin{remark}\label{rmk:conjgateordering}
If $<_\Gamma$ is the left-ordering on $\Mod(S)$ induced by a generalised ideal arc system $\Gamma$ for $S$, then the conjugate ordering $<^f_\Gamma$ is induced by the ideal arc system $f(\Gamma)$. In other words, the left-ordering $<^f_\Gamma$ is same as the left-ordering $<_{f(\Gamma)}$.
\end{remark} 

\begin{remark}\label{rmk:conjgateorderings}
The mapping class group $\Mod(S)$ admits a left-action on the set of generalised ideal arc systems for $S$ given by $(f, \Gamma) \mapsto f(\Gamma)$. In fact, this action is loose isotopy invariant, and hence induces a left-action on the set of loose isotopy classes of generalised ideal arc systems for $S$.
\end{remark} 

\begin{proposition}\label{conj left orderings and isotopy classes}
Two generalised ideal arc systems $\Gamma$ and $\Delta$ for $S$ induce conjugate left-orderings on $\Mod(S)$ if and only if $\Gamma$ and $\Delta$ are in the same orbit under the action of $\Mod(S)$ on the set of loose isotopy classes of generalised ideal arc systems.
\end{proposition}

\begin{proof}
Since the left-orderings induced by $\Gamma$ and $\Delta$ are conjugate, there exists an $f \in \Mod(S)$ such that $<_\Gamma$ is the same as $<^f_\Delta$. By Remark \ref{rmk:conjgateordering},  $<^f_\Delta$ is the same as $<_{f(\Delta)}$. By Theorem \ref{thm:loose2}, the generalised ideal arc systems $\Gamma$ and $f(\Delta)$ are loosely isotopic. Therefore, $\Gamma$ and $\Delta$ are in the same orbit under the action of $\Mod(S)$.
\par 
Conversely, let $g \in \Mod(S)$ such that $\Gamma$ and $g(\Delta)$ are loosely isotopic. By Proposition~\ref{prop:loose1}, the left-orderings induced by $\Gamma$ and $g(\Delta)$ on $\Mod(S)$ are the same. The assertion now follows from Remark~\ref{rmk:conjgateordering}.
\end{proof}

\begin{proposition}\label{space of orderings}
Let $S$ be a connected oriented infinite genus surface with non-empty boundary. Then the space of conjugacy classes of left-orderings on $\Mod(S)$ is infinite.
\end{proposition}

\begin{proof}
In view of Proposition \ref{conj left orderings and isotopy classes}, it suffices to prove that the quotient of the set of loose isotopy classes of generalised ideal arc systems under the action of $\Mod(S)$ is an infinite-set. Note that the surface $S$ admits a countable infinite number of separating ideal arcs $\{\eta_k\}_{k \ge 1}$ such that the finite-type component of $S \setminus \eta_k$ is distinct for each $k \ge 1$ (see Figure \ref{figLoch Ness}). For each $k \ge 1$, let $\Gamma_k$ be a generalised ideal arc system such that its first ideal arc is $\eta_k$. It is easy to see that for each $i\neq j$, the generalised ideal arc systems $\Gamma_i$ and $\Gamma_j$ lie in different orbits under the action of $\Mod(S)$.
\end{proof}
\begin{figure}[H]
	\includegraphics[width=2.4in]{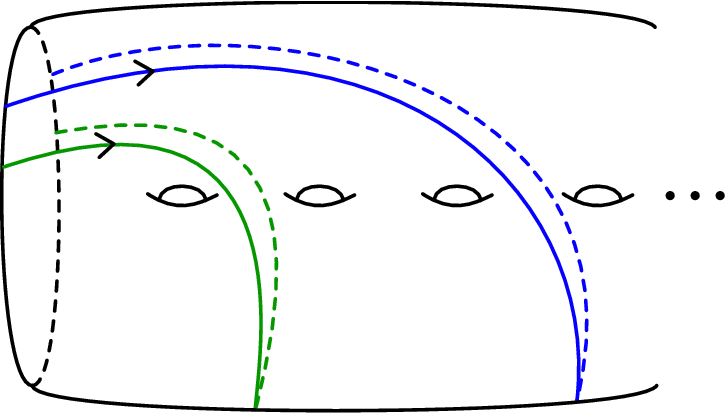}
	\caption{Countable infinite number of non-isotopic separating ideal arcs on the Loch Ness monster surface with one boundary component.}
		\label{figLoch Ness}
\end{figure}

\medskip

\section{Comparison of topologies on the big mapping class group}\label{section comparison of topologies}
In this section, we compare the quotient topology on the big mapping class group with the order topology induced by a left-ordering.

 \begin{definition}
    Let $(G, <)$ be a left-ordered group. The {\it order topology} on $G$ induced by $<$ is the topology for which the collection $\{L_g, R_g \mid g \in G \}$ forms a  sub-basis,  where $L_g=\{x\in G\mid x < g\}$ and $R_g=\{x\in G\mid  g< x\}$.
    \end{definition}
    
We note that, if the order topology on $G$ induced by the left-ordering $<$ is discrete, then the positive cone of $<$ has the  least element.

\begin{proposition}
Let $S$ be a connected oriented infinite-type surface with non-empty boundary and $\Gamma=\{\gamma_k \}_{k \ge 1}$ a (generalised) ideal arc system for $S$. Then the order topology on $\Mod(S)$ induced by the left-ordering $<_\Gamma$ is non-discrete.
\end{proposition}

\begin{proof}
It is enough to show that there is no least positive element in $\Mod(S)$ with respect to $<_\Gamma$. Suppose that $f \in \Mod(S)$ is the least positive element. Since $\id <_\Gamma f$, let $i$ be the smallest index such that $f(\gamma_i) \neq \gamma_i$ and  $f(\gamma_i)$ branches off $\gamma_i$ to the left. Since $\cup_{k=1}^i \gamma_k$ is a compact set, it follows that $\cup_{k=1}^i \gamma_k \subset S_\ell$ for some $\ell$. Choose an essential simple closed curve $c$  in $S\setminus S_\ell$.  If $T_c$ denote the left-handed Dehn twist along $c$, then $T_c$ is a positive element  with $T_c(\gamma_k)=\gamma_k$ for all $k \leq i$. Note that $i$ is the smallest index such that $T_c(\gamma_i) \ne f(\gamma_i)$ and $f(\gamma_i)$ branches off  $T_c(\gamma_i)$ to the left. Hence, $T_c <_\Gamma f$, which contradicts the fact that $f$ is the least positive element.
\end{proof}

Given a compact subset $K$ and an open subset $U$ of $S$, let $V(K,U)=\{f \in \Homeo^{+}(S,\partial S) \mid f(K) \subset U\}$. Then the collection $$\mathcal{B}=\{\cap_{l=1}^n V(K_l, U_l) \mid K~\textrm{compact,}~U~\textrm{open,}~ n\in \mathbb N\}$$ forms a basis for the compact open topology on $\Homeo^{+}(S,\partial S)$. Since the quotient map $q:\Homeo^{+}(S,\partial S)\to \Mod(S)$ is an open map, it follows that $q(\mathcal{B})$ is a basis for the quotient topology on $\Mod(S)$. Denoting $q(\cap_{l=1}^n V(K_l, U_l))$ by $\overline{\cap_{l=1}^n V(K_l, U_l)}$, we see that
 \begin{eqnarray*}
 \overline{\cap_{l=1}^n V(K_l, U_l)} &= &\{f \in \Mod(S)\mid f \text{ has a representative } f' \in \Homeo^+(S, \partial S) \text{ such that }\\ 
 & & f'(K_l)\subset U_l \text{ for all } 1 \leq l \leq n\}.
 \end{eqnarray*}

\begin{proposition}\label{prop:finer1}
Let $S$ be a connected oriented infinite-type surface with non-empty boundary and $\Gamma=\{\gamma_k \}_{k \ge 1}$ an ideal arc system for $S$. Then the quotient topology on $\Mod(S)$ is finer than the order topology induced by the left-ordering $<_\Gamma$.
\end{proposition}

\begin{proof}
By \cite[Lemma 13.3]{MR3728284}, it is enough to show that for each $f \in \Mod(S)$ and each open interval $(\phi, \psi)$ containing $f$, there exists a basis element $\overline{\cap_{l=1}^n V(K_l, U_l)}$ of the quotient topology such that $f\in \overline{\cap_{l=1}^n V(K_l, U_l)}\subseteq (\phi, \psi)$. Since $\phi <_\Gamma f <_\Gamma \psi$, there are smallest indices  $i$ and $j$ such that $\phi(\gamma_i) \neq f(\gamma_i)$ and $f(\gamma_j)\neq \psi (\gamma_j)$. For each $1 \le l \leq \max\{i,j\}$, let $K_l=\gamma_l$ and $U_l=f'(N(K_l))$, where $N(K_l)$ is a regular neighbourhood of $K_l$ and $f' \in \Homeo^{+}(S,\partial S)$ is a representative of the mapping class $f$. Since $f' \in \cap_{l=1}^{\max\{i,j\}} V(K_l, U_l)$, it follows that $f \in \overline{\cap_{l=1}^{\max\{i,j\}} V(K_l, U_l)}$. We claim that $\overline{\cap_{l=1}^{\max\{i,j\}} V(K_l, U_l)} \subseteq (\phi, \psi)$. If $g \in \overline{\cap_{l=1}^{\max\{i,j\}} V(K_l, U_l)}$, then it has a representative $g' \in \Homeo^{+}(S,\partial S) $ such that $g'(\gamma_l)$ is isotopic to $f'(\gamma_l)$ for all $1 \le l \leq \max\{i,j\}$. This implies that  $\phi(\gamma_k) = f(\gamma_k) = g(\gamma_k)$ for all $k \leq i-1$ and $f(\gamma_i) = g(\gamma_i)$ branches off $\phi(\gamma_i)$ to the left. Similarly, $\psi(\gamma_k) = f(\gamma_k) = g(\gamma_k)$ for all $k \leq j-1$ and $\psi(\gamma_j)$ branches off $ f(\gamma_j) = g(\gamma_j) $ to the left.
Hence, $g\in (\phi, \psi)$, and the proof is complete. 
\end{proof}

In the reverse direction, we have the following result.

\begin{proposition}\label{prop:finer2}
Let $S$ be a connected oriented infinite-type surface with non-empty boundary and $\Gamma=\{\gamma_k \}_{k \ge 1}$ the ideal arc system for $S$ as in \eqref{explicit stable Alexander system gamma}. Then the order topology induced by the left-ordering $<_\Gamma$  is finer than the quotient topology on $\Mod(S)$.
\end{proposition}

\begin{proof}
 Let $\{S_k \}_{k \ge 1}$ be the sequence of finite-type  subsurfaces of $S$ that satisfies the conditions of Proposition \ref{prop:cmpexh} and  $\Gamma=\{\gamma_k \}_{k \ge 1}$ the ideal arc system as in \eqref{explicit stable Alexander system gamma}. It is enough to show that for each $f \in \Mod(S)$ and each basis $\overline{\cap_{l=1}^n V(K_l, U_l)}$ containing $f$, there exists an open interval $(\phi, \psi)$ such that $f\in (\phi, \psi) \subset \overline{\cap_{l=1}^n V(K_l, U_l)}$. Let $S_{k}$ be the subsurface such that $K_l \subset S_{k}$  for each $1\leq l \leq n$. It follows from the construction of $\Gamma$ that the collection $S_k \cap (\cup_{\gamma\in \Gamma_{k+1}} \gamma)$ is a finite stable Alexander system for $S_{k}$.  Choose $\phi, \psi \in \Mod(S)$ such that $\phi< f< \psi$ and $\phi(\gamma_j)=\psi(\gamma_j)=f(\gamma_j)$ for all $\gamma_j \in \Gamma_{k +1}$. Then, for each $\eta \in (\phi, \psi)$, we have $\phi(\gamma_j)=\psi(\gamma_j)=f(\gamma_j)=\eta(\gamma_j)$ for all  $\gamma_j \in \Gamma_{k +1}$. Thus, there are representatives $f'$ and $\eta'$ of $f$ and $\eta$, respectively, such that $\eta' = f'$ on $S_{k}$, and hence $\eta'(K_l)=f'(K_l)\subset U_l$ for all $1\leq l\leq n$. This implies that $\eta \in \overline{\cap_{l=1}^n V(K_l, U_l)}$, and the proof is complete.
\end{proof}

Propositions \ref{prop:finer1}~and~\ref{prop:finer2} leads to the following result.

\begin{theorem}\label{quotient topology is order induced}
Let $S$ be a connected oriented infinite-type surface with non-empty boundary and let $\Gamma$ be the ideal arc system for $S$ as in \eqref{explicit stable Alexander system gamma}. Then the quotient topology on $\Mod(S)$ is the same as the order topology on $\Mod(S)$ induced by $<_\Gamma$.
\end{theorem}
\medskip

\begin{ack}
We would like to thank Diana Hubbard for bringing to our attention both her work on the left-orderability of the big mapping class group \cite{bigsmall} and the work of Calegari \cite{MR2172491}. Pravin Kumar is supported by the PMRF fellowship at IISER Mohali. Mahender Singh is supported by the Swarna Jayanti Fellowship grants DST/SJF/MSA-02/2018-19 and SB/SJF/2019-20/04. Apeksha Sanghi acknowledges post-doctoral fellowship from the grant SB/SJF/2019-20/04. 
\end{ack}

\section{Declaration}
The authors declare that they have no conflicts of interest and that there is no data associated with this paper.

\bibliographystyle{plain}

\begin{thebibliography}{1}

\bibitem{MR4584864} S. Agrawal, T. Aougab, Y. Chandran, M.  Loving, J. R. Oakley, R. Shapiro and Y. Xiao,  \textit{ Automorphisms of the k-curve graph}, Michigan Math. J. 73 (2023), no. 2, 305--343.

\bibitem{MR4264585} J. Aramayona and N. Vlamis, \textit{ Big mapping class groups: an overview}, In the tradition of Thurston--geometry and topology, 459--496, Springer, Cham, (2020).

\bibitem{arXiv:2103.16702} C. J. Bishop and  L. Rempe,  \textit{Non-compact Riemann surfaces are equilaterally triangulable}, arXiv:2103.16702.

\bibitem{MR2141698} S. Boyer, D. Rolfsen and B. Wiest, \textit{Orderable 3-manifold groups}, Ann. Inst. Fourier (Grenoble) 55 (2005), no. 1, 243--288.

\bibitem{MR2172491} D. Calegari, \textit{Circular groups, planar groups, and the Euler class}, Proceedings of the Casson Fest, 431--491, Geom. Topol. Monogr., 7, Geom. Topol. Publ., Coventry, 2004.

\bibitem{MR3560661} A. Clay and D. Rolfsen, \textit{Ordered groups and topology}, Graduate Studies in Mathematics, 176. American Mathematical Society, Providence, RI, 2016. x+154 pp.	

\bibitem{MR1214782} P. Dehornoy, \textit{Braid groups and left distributive operations}, Trans. Amer. Math. Soc. 345 (1994), no. 1, 115--150.

\bibitem{MR2463428} P. Dehornoy, I. Dynnikov, D. Rolfsen and B. Wiest, \textit{Ordering braids}, Mathematical Surveys and Monographs, 148, American Mathematical Society, Providence, RI, 2008, x+323 pp.

\bibitem{MR0975081} M. Falk and R. Randell, \textit{Pure braid groups and products of free groups}, Contemp. Math. 78 (1988), 217--228.

\bibitem{MR2850125} B. Farb and D. Margalit, \textit{A primer on mapping class groups}, Princeton Mathematical Series, 49. Princeton University Press, Princeton, NJ, 2012. xiv+472 pp.

\bibitem{bigsmall} P. Feller, D. Hubbard and H. Turner, \textit{The Dehn twist coefficient for big and small mapping class groups}, (2023), https://arxiv.org/abs/2308.06214.

\bibitem{MR1725462}  R. Fenn, M. T. Greene, D. Rolfsen, C. Rourke and B. Wiest, \textit{Ordering the braid groups}, Pacific J. Math. 191 (1999), no. 1, 49--74.


\bibitem{MR4029627} J. Hern\'{a}ndez Hern\'{a}ndez, I. Morales and F. Valdez,  \textit{The Alexander method for infinite-type surfaces}, Michigan Math. J. 68 (2019), no.4, 743--753.

\bibitem{MR3728284} J. R. Munkres, \textit{ Topology}, {Second edition}, Prentice Hall, Inc., Upper Saddle River, NJ, 2000. xvi+537 pp.

\bibitem{MR1990838} B. Perron and D. Rolfsen, \textit{On orderability of fibred knot groups}, Math. Proc. Cambridge Philos. Soc. 135 (2003), no. 1, 147--153.

\bibitem{Rado} T. Rad\'{o}, \textit{\"{U}ber den Begriff der Riemannschen Fl\"{a}che}, Acta Litt. Sci. Szeged 2 (1925), 101--121. 

\bibitem{MR1756636} C. Rourke and B. Wiest, \textit{Order automatic mapping class groups}, Pacific J. Math. 194 (2000), no. 1, 209--227.			

\bibitem{MR1805402}  H. Short and B. Wiest, \textit{ Orderings of mapping class groups after Thurston}, Enseign. Math. 46 (2000), no.3--4, 279--312.

\end{thebibliography}

\bigskip

 \end{document}